\documentclass[11pt,a4paper]{article}
\usepackage{appendix}        \usepackage[alphabetic,lite]{amsrefs}                                                       
\usepackage{array}                                                             
\usepackage{bm}                                                                
\usepackage{colonequals}                                                       
\usepackage{color}                                                             
\usepackage{dynkin-diagrams}                                                   
\usepackage{fullpage}                                                          
\usepackage{indentfirst}                             
\usepackage{moreenum}                                                          
\usepackage{stmaryrd}                                                          
\usepackage{verbatim}                                                          
\usepackage[all,cmtip]{xy}                                                     

\usepackage{amssymb}                                                           
\usepackage{amsthm}                                             

\usepackage{graphicx}       \usepackage{tikz-cd}                                                
\usepackage{mathrsfs}                                                          
\usepackage{mathtools}                         
\usepackage{wrapfig}
\usepackage[all,cmtip]{xy}                                                     

\usepackage[OT2,T1]{fontenc}
\DeclareSymbolFont{cyrletters}{OT2}{wncyr}{m}{n}
\DeclareMathSymbol{\Be}{\mathalpha}{cyrletters}{"42}                          
\DeclareMathSymbol{\Che}{\mathalpha}{cyrletters}{"51}                          
\DeclareMathSymbol{\Sha}{\mathalpha}{cyrletters}{"58}                          

\makeatletter
\DeclareRobustCommand\widecheck[1]{{\mathpalette\@widecheck{#1}}}
\def\@widecheck#1#2{%
    \setbox\z@\hbox{\m@th$#1#2$}%
    \setbox\tw@\hbox{\m@th$#1%
       \widehat{%
          \vrule\@width\z@\@height\ht\z@
          \vrule\@height\z@\@width\wd\z@}$}%
    \dp\tw@-\ht\z@
    \@tempdima\ht\z@ \advance\@tempdima2\ht\tw@ \divide\@tempdima\thr@@
    \setbox\tw@\hbox{%
       \raise\@tempdima\hbox{\scalebox{1}[-1]{\lower\@tempdima\box
\tw@}}}%
    {\ooalign{\box\tw@ \cr \box\z@}}}
\makeatother

\theoremstyle{plain}      \newtheorem{thm}{Theorem}[section]                   
\theoremstyle{plain}                   
\theoremstyle{plain}      \newtheorem{lem}[thm]{Lemma}                         
\theoremstyle{plain}                             
\theoremstyle{plain}      \newtheorem{cor}[thm]{Corollary}                     
\theoremstyle{plain}                         
\theoremstyle{plain}      \newtheorem{prop}[thm]{Proposition}                  
\theoremstyle{plain}      \newtheorem{conjecture}[thm]{Conjecture}             
\theoremstyle{definition} \newtheorem{rmk}[thm]{Remark}                        
\theoremstyle{definition}                      
\theoremstyle{definition} \newtheorem{df}[thm]{Definition}                     
\theoremstyle{definition}                   
\theoremstyle{definition} \newtheorem{eg}[thm]{Example}                        
\theoremstyle{definition}                        
\theoremstyle{definition}                        
\theoremstyle{definition}                      
\theoremstyle{definition}                    
\theoremstyle{definition}                  
\theoremstyle{definition}                        
\theoremstyle{definition}                       
\theoremstyle{definition}                    
\theoremstyle{definition}                  
\theoremstyle{definition} \newtheorem{construction}[thm]{Construction}         
\theoremstyle{definition}              
\theoremstyle{definition}                  
\theoremstyle{definition} \newtheorem{prop-df}[thm]{Proposition-Definition}    
\theoremstyle{definition} \newtheorem{thm-df}[thm]{Theorem-Definition}
\theoremstyle{definition}
\newtheorem*{construction*}{Construction}                                      
\newtheorem*{conjecture*}{Conjecture}                                          
\newtheorem*{hypothesis*}{Hypothesis}                                          
\newtheorem*{convention*}{Convention}                                          
\newtheorem*{notation*}{Notation}                                              
\newtheorem*{prop*}{Proposition}                                               
\newtheorem*{summary*}{Summary}                                                
\newtheorem*{qt*}{Question}                                                    
\newtheorem*{rmk*}{Remark}                                                     
\newtheorem*{fact*}{Fact}                                                      
\newtheorem*{lizi*}{Example}                                                   
\newtheorem*{df*}{Definition}                                                  
\theoremstyle{plain}
\newtheorem*{thm*}{Theorem}                                                    

\newcommand{\bconst}{\begin{construction}}                                     
\newcommand{\econst}{\end{construction}}                                       
\newcommand{\benum}{\begin{enumerate}[label={{\upshape(\alph*)}}]}             
\newcommand{\benuma}{\begin{enumerate}[label={{\upshape(\arabic*)}}]}          
\newcommand{\benumr}{\begin{enumerate}[label={{\upshape(\roman*)}}]}           
\newcommand{\eenum}{\end{enumerate}}
\newcommand{\bconj}{\begin{conjecture}}
\newcommand{\econj}{\end{conjecture}}
\newcommand{\bconjnn}{\begin{conjecture*}}
\newcommand{\econjnn}{\end{conjecture*}}
\newcommand{\begs}{\begin{eg}\hfill\benuma}                                    
\newcommand{\eegs}{\eenum\end{eg}}                                             
\newcommand{\brmks}{\begin{rmk}\hfill\benuma}                                  
\newcommand{\ermks}{\eenum\end{rmk}}                                           
\newcommand{\bdfs}{\begin{df}\hfill\benuma}                                    
\newcommand{\edfs}{\eenum\end{df}}                                             
\newcommand{\bitem}{\begin{itemize}}                                           
\newcommand{\eitem}{\end{itemize}}                                             
\newcommand{\be}{\begin{equation}}                                             
\newcommand{\ee}{\end{equation}}                                               
\newcommand{\benn}{\begin{equation*}}                                          
\newcommand{\eenn}{\end{equation*}}                                            
\newcommand{\bqt}{\begin{qt*}\rm}                                              
\newcommand{\eqt}{\end{qt*}}                                                   
\newcommand{\bqtr}{\begin{qt*}\rm\coLR}                                        
\newcommand{\eqtr}{\end{qt*}}                                                  
\newcommand{\beac}{\begin{equation}\begin{array}{c}}                           
\newcommand{\eeac}{\end{array}\end{equation}}                                  
\newcommand{\beqn}{\begin{eqnarray*}}
\newcommand{\eeqn}{\end{eqnarray*}}
\newcommand{\bdf}{\begin{df}}
\newcommand{\bdfhf}{\begin{df}\hfill}
\newcommand{\edf}{\end{df}}
\newcommand{\brmk}{\begin{rmk}}
\newcommand{\brmkhf}{\begin{rmk}\hfill}
\newcommand{\ermk}{\end{rmk}}

\newcommand{\BA}{\mathbf{A}}

   \newcommand{\BL}{\mathbf{L}}
   
   \newcommand{\BP}{\mathbf{P}}

  \newcommand{\CH}{\mathcal{H}}

\newcommand{\BBC}{\mathbb{C}}  
  
\newcommand{\BBG}{\mathbb{G}}

\newcommand{\BBQ}{\mathbb{Q}}

  \newcommand{\BBZ}{\mathbb{Z}}

\renewcommand{\P}{\mathbb{P}}                                                  
\newcommand{\Z}{\mathbb{Z}}                                                    
\newcommand{\QZ}{\mathbb{Q}/\mathbb{Z}}                                        
\newcommand{\Qp}{\mathbb{Q}_{p}}                                               
\newcommand{\Zp}{\mathbb{Z}_{p}}                                               
\newcommand{\coLR}{\textcolor[rgb]{1.00,0,0}}                                  

\DeclareMathOperator{\lcm}{lcm}	
\DeclareMathOperator{\ssym}{SSym}		

\DeclareMathOperator{\Hom}{Hom}                                                
\DeclareMathOperator{\ext}{Ext}                                                

\DeclareMathOperator{\Div}{Div}                                                
\DeclareMathOperator{\Pic}{Pic}                                                
\DeclareMathOperator{\br}{Br}                                                  

\DeclareMathOperator{\Image}{Im}                                               
\renewcommand{\Im}{\Image}                                                     
\DeclareMathOperator{\Ker}{Ker}                                                
\renewcommand{\ker}{\Ker}                                                      

\DeclareMathOperator{\Id}{Id}                                                  
\DeclareMathOperator{\Nrd}{Nrd}                                                

\DeclareMathOperator{\Spec}{Spec}                                              
\DeclareMathOperator{\Supp}{Supp}                                              

\newcommand{\tors}{\mathrm{tors}}                                              
\newcommand{\cts}{\mathrm{cont}}                                               

\newcommand{\nr}{{\mathrm{nr}}}                                                
\newcommand{\Zar}{\mathrm{Zar}}                                                
\newcommand{\et}{\mathrm{\acute{e}t}}                                          

\DeclareMathOperator{\Inv}{Inv}                                                

\DeclareMathOperator{\CHOW}{CH}                                                

\newcommand{\gm}{\BBG_m}                                                       
\DeclareMathOperator{\GL}{\mathbf{GL}}                                         
\DeclareMathOperator{\SL}{\mathbf{SL}}                                         
\DeclareMathOperator{\PGL}{\mathbf{PGL}}                                       
\DeclareMathOperator{\SU}{\mathbf{SU}}                                         
\DeclareMathOperator{\Sp}{\mathbf{Sp}}                                         
\DeclareMathOperator{\Res}{Res}                                                
\DeclareMathOperator{\Inf}{Inf}                                                
    
\DeclareMathOperator{\Cor}{Cor}   

\newcommand{\hz}{H_{\Zar}}                                                     

\newcommand{\rmnorm}{\mathrm{norm}}

\theoremstyle{definition}

\newcommand\addtag{\refstepcounter{equation}\tag{\theequation}}

\definecolor{tianred}{rgb}{0.57, 0.36, 0.51}                                   
\definecolor{tianblue}{rgb}{0.0, 0.22, 0.66}                                   
\definecolor{tianpink}{rgb}{0.88, 0.56, 0.59}                                  
\definecolor{tiangreen}{rgb}{0.24, 0.82, 0.44}                                 
\definecolor{col1}{HTML}{74664D}
\definecolor{col3}{HTML}{CEA964}
\definecolor{col2}{HTML}{857251}
\definecolor{col4}{HTML}{F7F4EF}

\usepackage[colorlinks,
            linkcolor=tianred,
            anchorcolor=col3,
            citecolor=tianblue,
            urlcolor=col3]{hyperref}                                       

\title{Reciprocity obstruction to strong approximation over $p$-adic function fields}
\author{Haowen Zhang}
\date{}
\usepackage[top=2cm, bottom=2.5cm, left=2.2cm, right=2.2cm]{geometry}
\begin{document}
\maketitle
\begin{abstract}
Over function fields of $p$-adic curves, we construct stably rational varieties in the form of homogeneous spaces of $\SL_n$ with semisimple simply connected stabilizers and we show that strong approximation away from a non-empty set of places fails for such varieties. The construction combines the Lichtenbaum duality and the degree $3$ cohomological invariants of the stabilizers. We then establish a reciprocity obstruction which accounts for this failure of strong approximation. We show that this reciprocity obstruction to strong approximation is the only one for counterexamples we constructed, and also for classifying varieties of tori. We also show that this reciprocity obstruction to strong approximation is compatible with known results for tori. At the end, we explain how a similar point of view shows that the reciprocity obstruction to weak approximation is the only one for classifying varieties of tori over $p$-adic function fields.
\end{abstract}
\section{Introduction}
In recent years, there has been growing interest in studying the arithmetic of algebraic groups and their homogeneous spaces defined over fields of cohomological dimension strictly greater than $2$, and in particular, function fields $K$ of $p$-adic curves. For example, in \cite{harari2013local} and \cite{harari2015weak}, Harari, Szamuely and Scheiderer studied local-global principle and weak approximation questions for tori over such fields $K$, with respect to the valuations coming from the codimension $1$ points of the $p$-adic curve. For both questions, they showed that the reciprocity obstruction is the only one. Tian generalized their results on weak approximation to quasi-split reductive groups in his work \cite{tian2021obstructions}. Very recently in \cite{linh2022arithmetics}, Linh studied homogeneous spaces of $\SL_n$ with toric stabilizers, and more generally, stabilizers in the form of an extension of a group of multiplicative type by a unipotent group. He showed that the reciprocity obstruction is the only one to both the local-global principle and weak approximation. 
\par Being a generalized form of the Chinese Remainder Theorem, strong approximation is a natural question in the study of the arithmetic of algebraic groups and their homogeneous spaces. For tori over $p$-adic function fields, Harari and Izquierdo described the defect of strong approximation in \cite{harari2019espace}, based on a Poitou-Tate type exact sequence obtained in \cite{harari2015weak}.
\par We study strong approximation questions for homogeneous spaces over $p$-adic function fields. Let $K$ be the function field of a smooth proper geometrically integral curve $X$ defined over a $p$-adic field $k.$ We first show that we have very different behavior of strong approximation over $K$ compared to number fields:
\begin{thm}There are homogeneous spaces of the form $Z=\SL_n/H$ with $H$ semisimple simple connected such that strong approximation away from a non-empty set $S\subseteq X^{(1)}$ does not hold for $Z$.
\end{thm}
Our constructions include $H$ of inner type $\mathsf A$ (Theorem \ref{thm failure of sa}, Corollary \ref{cor failure SA SLn/SL1A A constant} and Example \ref{example}), and outer type $\mathsf A$ (Corollary \ref{cor outer type} and Example \ref{example outer type}). 
\par For $H$ of inner type $\mathsf A$, we have $H=\SL_1(A)$ for a simple central $K$-algebra $A$ and $Z=\SL_n/H$ is a stably rational $K$-variety. We show that the maps $H^1(K_v,H)\rightarrow H^3(K_v,\QZ(2))$ given by the Rost invariant composed with the sum in the complex of the generalized Weil reciprocity law can be calculated in terms of the Lichtenbaum duality pairing $\br X\times \Pic X\rightarrow \QZ$ induced by evaluation on closed points. As a consequence, we give explicit descriptions of $H^1(K_v,H)$, which are finite cyclic groups, and conditions under which the map $H^1(K,H)\rightarrow \bigoplus_{v\in X^{(1)}\backslash S}H^1(K_v,H)$ is not surjective, yielding counter-examples to strong approximation for $\SL_n/H$. 
\par For $H$ of outer type $\mathsf A$, we have $H=\SU(A,\tau)$ for a central simple $L$-algebra $A$ with $L/K$-unitary involution $\tau$, where $L/K$ is a separable quadratic extension. We describe the images of $H^1(K_v,H)$ under the Rost invariant with respect to different ramification types, and we give counter-examples to strong approximation for $\SL_n/H$ based on calculations in the case of inner type $\mathsf A$.
\par To account for such failure, we define a reciprocity obstruction. We use the group $H^3_{\nr}(Z,\QZ(2))$, which is the subgroup of elements in $H^3(K(Z),\QZ(2))$ that are unramified with respect to all codimension $1$ points of $Z$. We define the subgroup of ``trivial on $S$'' elements
\begin{equation*}H^3_{\nr,S}(Z,\QZ(2)):=\ker(H^3_\nr(Z,\QZ(2))
\rightarrow\prod_{v\in S}H^3_\nr(Z_{K_v},\QZ(2)))\end{equation*}
and we denote by $\overline H^3_{\nr,S}(Z,\QZ(2))$ its quotient by constant elements. We define a pairing
\[Z(\BA_K^S)\times \overline H^3_\nr(Z,\QZ(2))\rightarrow\QZ,\quad ((x_v)_{v\notin S},\alpha)\mapsto\sum_{v\notin S}\alpha(x_v).\]
The subset of elements in $Z(\BA_K^S)$ that are orthogonal to $\overline H^3_{\nr,S}(Z,\QZ(2))$ contains $Z(K)$ in virtue of the generalized Weil reciprocity law, and also its closure $\overline{Z(K)}^S$ by continuity of the pairing, giving rise to a reciprocity obstruction to strong approximation away from $S$.
\par When $Z$ is of the form $\SL_n/H$, we show that we can also define a pairing using $\Inv^3(H,\QZ(2))_\rmnorm$, the normalized degree $3$ cohomological invariants of $H$ with coefficients in $\QZ(2)$, and we have the following compatibility given be the commutative diagram 
\[
\begin{tikzcd}
Z(\BA_K^S) \arrow[d,Rightarrow, no head] \arrow[r,phantom,"\times" description] & {\Inv^3(H,\QZ(2))_\rmnorm} \arrow[r] \arrow[d, "\simeq" ]       & \QZ \arrow[d,Rightarrow, no head] \\
Z(\BA_K^S)                                \arrow[r,phantom,"\times" description] & {\overline H^3_\nr(Z,\QZ(2))} \arrow[r]  & \QZ.                               
\end{tikzcd}
\]
As a consequence, there is a reciprocity obstruction to strong approximation for the examples we constructed. Then we show that for certain $Z$ (e.g. the one in Example \ref{example}, or $\SL_n/T$ for $T$ a torus), this obstruction is the only one.
\begin{thm}[Theorem \ref{thm rec only one sc}]
Let $Z=E/\SL_1(A)$ with $[A]\in\br X$ of exponent $m$, and the special rational group $E$ is split semisimple simply connected. For the $p$-adic curve $X$, suppose $\Pic^0(X)/m=0$, or equivalently $\prescript{}{m} H^1(k,\Pic\overline X)=0$ (for example $X=\P^1_k$ the projective line satisfies this condition). Then there is an exact sequence of pointed sets
$$1\rightarrow \overline{Z(K)}^S\rightarrow Z(\BA_K^S)\rightarrow \overline H^3_{\nr,S}(Z,\QZ(2))^D
\rightarrow 1$$ for $S\subseteq X^{(1)}$ a non-empty finite set of places. In particular,
 the reciprocity obstruction to strong approximation away from $S$ is the only one for $Z$. 
 \par The group $\overline H^3_{\nr,S}(Z,\QZ(2))$ measures the defect of strong approximation away from $S$ for $Z$, and is finite cyclic of order $\gcd(\frac{I(S)}{I(X)},m)$ where $I(X)$ (resp. $I(S)$) is the index of $X$ (resp. $S$). In particular, strong approximation away from $S$ holds for $Z$ if and only of $I(S)/I(X)$ is coprime to $m$, and such an $S$ always exists, for example $S$ such that $I(S)=I(X)$.
\end{thm}
 The following result is obtained by exploiting the Poitou-Tate type exact sequence in \cite{harari2015weak} along with the description of $\Inv^3(T,\QZ(2))_\rmnorm$ given by Blinstein and Merkurjev in \cite{blinstein2013cohomological}.
\begin{thm}[Theorem \ref{thm rec only one for SLn/tori}]Let $Z=E/T$ be a classifying variety of a torus $T$ over $K$ a $p$-adic function field, where the special rational group $E$ is split semisimple simply connected. Then there is an exact sequence of pointed sets
\[1\rightarrow \overline{Z(K)}^S\rightarrow Z(\BA_K^S)\rightarrow \overline H^3_{\nr,S}(Z,\QZ(2))^D\] for $S\subseteq X^{(1)}$ a non-empty finite set of places. In other words,
the reciprocity obstruction to strong approximation away from $S$ is the only one for $Z$. If we suppose furthermore that $\CHOW^2(Z)\rightarrow H^0(K,\CHOW^2(\overline Z))$ is surjective, then there is an exact sequence of pointed sets
\[1\rightarrow \overline{Z(K)}^S\rightarrow Z(\BA_K^S)\rightarrow \overline H^3_{\nr,S}(Z,\QZ(2))^D\rightarrow \overline H^3_{\nr,X^{(1)}}(Z,\QZ(2))^D\rightarrow 1.
\]
\end{thm}
Then we show that our reciprocity obstruction to strong approximation not only works for classifying varieties, but also applies to tori, compatible with the results in \cite{harari2019espace} obtained by Harari and Izquierdo:
\begin{thm}[Corollary \ref{cor SA tori}]
The morphism $T(\BA_K)\rightarrow H^3_{\nr}(T,\QZ(2))^D$ induces injective morphisms
\[{A(T)/\Div}\hookrightarrow H^3_{\nr}(T,\QZ(2))^D,\]
\[A(T)_\tors\hookrightarrow ({H^3_{\nr}(T,\QZ(2))_\wedge})^D\]
where $H^3_{\nr}(T,\QZ(2))_\wedge:=\varprojlim_m{H^3_{\nr}(T,\QZ(2))/m},$ the group $A(T):=T(\BA_K)/\overline{T(K)}=T(\BA_K)/T(K)$ measures the defect of strong approximation, and $A(T)/\Div$ denotes the quotient of $A(T)$ by its maximal divisible subgroup.
\end{thm}
Finally, we explain that for classifying varieties, our comparison between the two pairings of reciprocity obstruction and cohomological invariants has a parallel form for weak approximation problems too, given by the commutative diagram 
\[
\begin{tikzcd}
\prod_{v\in X^{(1)}}Z(K_v) \arrow[d,Rightarrow, no head] \arrow[r,phantom,"\times" description] & {\Inv_\nr^3(H,\QZ(2))_\rmnorm} \arrow[r]        & \QZ \arrow[d,Rightarrow, no head] \\
\prod_{v\in X^{(1)}}Z(K_v)                              \arrow[r,phantom,"\times" description] & {H^3_\nr(K(Z)/K,\QZ(2))/\Im(H^3(K,\QZ(2)))} \arrow[r] \arrow[u, leftarrow,"\simeq" ] & \QZ
\end{tikzcd}\addtag\label{intro WA pairing diagram}\]
where $\Inv^3_\nr(H,\QZ(2))_\rmnorm$ denotes the subgroup of $\Inv^3(H,\QZ(2))_\rmnorm$ consisting of the unramified invariants. 
\par As a quick and direct application, we obtain the following theorem giving answers to weak approximation problems for classifying varieties of tori over $K$, combining the Poitou-Tate type exact sequence in \cite{harari2015weak} and the description of $\Inv_\nr ^3(T,\QZ(2))_\rmnorm$ in \cite{blinstein2013cohomological}. This result was also obtained by Linh in a different way in his very recent work (cf. Theorem B of \cite{linh2022arithmetics}).
\begin{thm}[Theorem \ref{thm WA SLn/T}]
    Let $Z=E/T$ be a classifying variety of a torus $T$ over $K$. Then the reciprocity obstruction to weak approximation is the only one for $Z$. In fact, there is a morphism $$\Sha_S^1(K,T^\prime)\rightarrow H^3_\nr(K(Z)/K,\QZ(2))/\Im(H^3(K,\QZ(2)))$$ such that the subset of elements in $\prod_{v\in S}Z(K_v)$ orthogonal to the image of $\Sha_S^1(K,T^\prime)$ in $\frac{H^3_\nr(K(Z)/K,\QZ(2))}{\Im(H^3(K,\QZ(2)))} $ with respect to the pairing (\ref{intro WA pairing diagram}) already equals the closure $\overline {Z(K)}$ in the topological product $\prod_{v\in S}Z(K_v)$, where $\Sha^1_S(T^\prime):=\Ker(H^1(K,T^\prime)\rightarrow\prod_{v
\in X^{(1)}\backslash S}H^1(K_v,T^\prime))$ for a finite set $S$.
\end{thm}

\section*{Preliminaries}
\subsection*{Motivic complexes}
Let $X$ be a smooth scheme. Lichtenbaum (cf. \cite{lichtenbaum1987construction}\cite{lichtenbaum2007new}) defined the motivic complexes $\BBZ(i)$ for $i=0,1,2,$ of étale sheaves on $X$. We write $H^*(X,\BBZ(i))$ for the étale (hyper)cohomology groups of $X$ with values in $\BBZ(i)$. The complex $\BBZ(0)$ equals the constant sheaf $\BBZ$ and $\BBZ(1)=\BBG_{m,X}[-1]$, hence $H^n(X,\BBZ(1))=H^{n-1}(X,\BBG_{m,X})$. In particular, $H^3(X, \BBZ(1)) = \br X$, the cohomological Brauer group of $X$. The complex $\BBZ(2)$ is concentrated in degrees $1$ and $2$ and there is a product map $\BBZ(1)\otimes^\BL\BBZ(1)\rightarrow\BBZ(2)$ (cf. Proposition 2.5 of \cite{lichtenbaum1987construction}). Over the small Zariski site $X_{\Zar}$, we have the complex $\Z(2)_\Zar$ concentrated in degree $\leq 2$ defined in a similar way to $\Z(2)$.
\par We denote by $A(i)$ the complex $A\otimes \Z(i)$ for an abelian group $A$ (similarly for $\Z(i)_\Zar$). If $X$ is defined over a field whose characteristic does not divide $m$, we have a quasi-isomorphism of complexes of étale sheaves $\BBZ/m\BBZ\otimes^{\BL}\BBZ(i)\simeq\mu_m^{\otimes i}$, 
where $\mu_m$ is the étale sheaf of $m$th roots of unity, and we set $\mu_m^0=\BBZ/m\BBZ$. Thus we shall also use the notation $\QZ(i)$ for the direct limit of the sheaves $\mu_m^{\otimes i}$ for all $m>0$ when the base field has characteristic $0$.
\par The exact triangle in the derived category of étale sheaves
\[\BBZ(i)\rightarrow\BBQ(i)\rightarrow\QZ(i)\rightarrow\BBZ(i)[1]\]
yields the connecting morphism 
\[H^n(X,\QZ(i))\rightarrow H^{n+1}(X,\BBZ(i)),\]
which is an isomorphism if $X=\Spec F$ for a field $F$ and $n>i$ (cf. Lemma 1.1 of \cite{kahn1993descente}). 
\subsection*{Classifying varieties and cohomological invariants} A connected algebraic group $E$ defined over a field $K$ is called \textit{special} if $H^1(F,E)=1$ for all field extensions $F/K.$ Let $H$ be a smooth algebraic group over $K$. Choose an embedding $\rho:H\hookrightarrow E$ into a special rational group $E$. Examples of special rational groups $E$ include $\SL_n,\Sp_{2n},\GL_1(A)$ for a central simple $K$-algebra $A$.  The variety $Z=E/\rho(H)$ is called a \textit{classifying variety} of $H$. Different choices of $\rho$ give stably birational classifying varieties. (cf. \S2 of \cite{merkurjev2002unramified}.)
\par
For every field extension $F/K$, the set $H^1(F,H)$ classifying $H$-torsors over $F$ can be identified with $Z(F)/E(F)$, the orbit space of the action of $E(F)$ on $Z(F)$. For a point $x\in Z(F)$, we write $x^*E$ for the fiber of the $H$-torsor $E\rightarrow Z$ above the point $x$ and it is an $H$-torsor over $F$, and its class in $H^1(F,H)$ is the image of $[E]$ under the map $H^1(Z,H)\rightarrow H^1(F,H)$ induced by $x$. In particular, we have the generic $H$-torsor $\xi^*E$ over $K(Z)$ corresponding to the generic point $\xi\in Z$.
\par
Let $H$ be an algebraic group over a field $K.$ The group $\Inv^d(H,\QZ(d-1))$ of degree $d$ cohomological invariants of $H$ consists of morphisms of functors \[H^1(*,H)\rightarrow H^d(*,\QZ(d-1))\] from the category of field extensions of $K$ to the category of sets. An invariant is called \textit{normalized} if it takes the trivial $H$-torsor to zero. The normalized invariants form a subgroup $\Inv^d(H,\QZ(d-1))_\rmnorm$ of $\Inv^d(H,\QZ(d-1))$ and there is a natural isomorphism
\[\Inv^d(H,\QZ(d-1))\simeq H^d(K,\QZ(d-1)\oplus\Inv^d(H,\QZ(d-1))_\rmnorm.\]
Rost proved that if $H$ is absolutely simple simply connected, then $\Inv^3(H,\QZ(2))_\rmnorm$ is a finite cyclic group generated by the \textit{Rost invariant} $R_H$ (cf. Theorem 9.11 of \cite{garibaldi2003cohomological}).

\subsection*{$p$-adic function fields}
Let $k$ be a finite extension of $\Qp$, and let $X$ be a smooth proper geometrically integral curve over $k$. Let $K$ be the function field of $X$. There is a natural set of places: the set of all closed points of $X$ that we denote by $X^{(1)}$. For a closed point $v\in X^{(1)},$ we denote by $\kappa(v)$ its residue field, and $K_v$ the completion of $K$ for the discrete valuation induced by $v$. Then $\kappa(v)$ is also the residue field of the ring of integers $\mathcal O_v$ and is a finite extension of $k$. Therefore $K_v$ is a \textit{2-dimensional local fields}, i.e. a field complete with respect to a discrete valuation whose residue field is a classical local field.\par
We have a complex in virtue of the generalized Weil reciprocity law (cf. Theorem 6.4.3 of \cite{gille2017central}):
\begin{equation}\label{Weil}H^3(K,\BBQ/\BBZ(2))\rightarrow \bigoplus_{v\in X^{(1)}}H^3(K_v,\BBQ/\BBZ(2))\xrightarrow[]{\Sigma}\BBQ/\BBZ\end{equation}
where the map $\Sigma$ is given by the composition \[\bigoplus_{v\in X^{(1)}}H^3(K_v,\BBQ/\BBZ(2))\xrightarrow[\simeq]{\partial_v}\bigoplus_{v\in X^{(1)}}H^2(\kappa(v),\BBQ/\BBZ(1))\xrightarrow[]{\Sigma\Cor_{\kappa(v)/k}}H^2(k,\BBQ/\BBZ(1))=\BBQ/\BBZ.\]
The above notation of our fields will be fixed for the rest of the article.
\subsection*{Strong approximation}
Given a smooth geometrically integral $K$-scheme $Z$, for a non-empty open subscheme $U\subseteq X$ sufficiently small, we can find a smooth geometrically integral $U$-scheme $\mathcal Z$ such that $\mathcal Z\times_U K\simeq Z$. We can thus define the adelic space (which is actually independent of the choice of the model $\mathcal Z$):
\[Z(\BA_K):=\varinjlim_{\substack{U'\subseteq  U \\U'\neq\emptyset}}\prod_{v\in X^{(1)}\backslash U'}Z(K_v)\times\prod_{v\in U'^{(1)}}\mathcal Z (\mathcal O_v).\]
The problem of strong approximation studies the closure $\overline{Z(K)}$ inside the adelic space $Z(\BA_K)$ with the restricted product topology, possibly away from a finite set $S$ of places (in this case, we study $\overline{Z(K)}^S$ the closure of $Z(K)$ in $Z(\BA_K^S):=\prod_{v\notin S}^{\prime}Z(K_v)$ the restricted product with respect to integral points excluding places in $S$). If ${Z(K)}$ is dense in $Z(\BA_K^S)$, we say that $Z$ satisfies \textit{strong approximation away from $S$}: we can simultaneously approximate finitely
many $K_v$-points with $v\notin S$ by a single $K$-rational point with the condition that this
point is integral at all other places.
For $S\subseteq S^\prime$, strong approximation away from $S$ implies strong approximation away from $S^\prime$.

\par
Strong approximation for semisimple groups has been studied in the general setting. A $K$-group $G$ is said to be \textit{quasi-split} if it has a Borel subgroup defined over $K$. A $K$-group $G$ is called \textit{absolutely almost simple} if over the algebraic closure it becomes an extension of a simple group by a finite normal (hence central) subgroup.
\begin{thm}[Gille, Corollaire 5.1 of \cite{gille2009probleme}]
Let $R$ be a Dedekind domain with function field $K$ and let $G$ be a semisimple simply connected $K$-group which is absolutely almost simple and isotropic. Suppose the $K$-variety $G$ is retract rational. Then strong approximation holds for $G$ over $\Spec R$.
\end{thm}
The above theorem implies an earlier result:
\begin{thm}[Harder, Satz 2.2.1 of \cite{harder1967halbeinfache}]
Let $R$ be a Dedekind domain with function field $K$. Let $G$ be a quasi-split semisimple simply connected $K$-group. Then strong approximation holds for $G$ over $\Spec R$.
\end{thm}
\begin{cor}\label{split groups satisfy SA} Let $X$ be a smooth proper geometrically integral curve over a $p$-adic field $k$, with function field $K=k(X)$. Let $G$ be a quasi-split semisimple simply connected $K$-group. Then strong approximation away from a non-empty finite set of places $S\subseteq X^{(1)}$ holds for $G$.
\end{cor}
\begin{proof}
The open subset $V:=X\backslash S$ is affine and the ring of regular functions $k[V]$ is a Dedekind domain. 
\end{proof}
\par
When studying strong approximation for classifying varieties $Z=E/H$, we usually make use of the following commutative diagram 
\[\begin{tikzcd}
    E(K)\arrow[r]\arrow[d]&Z(K)\arrow[r]\arrow[d]&H^1(K,H)\arrow[r]\arrow[d]&1\\
    E(\BA_K^S)\arrow[r]&Z(\BA_K^S)\arrow[r]& \BP_S^1(H)\arrow[r]&1
\end{tikzcd}\]
where \[\BP_S^i(H):={\prod_{v\not\in S}}' H^i(K_v,H)\addtag\label{p1}\]
denotes the restricted topological product of the $H^i(K_v,H)$ for all $v\in X^{(1)}\backslash S$ with respect to the images of the maps $H^i(\mathcal O_v,\mathcal H)$ for $v\in U$.

\subsection*{Notation}
Unless otherwise stated, all (hyper)cohomology groups are taken with respect to the étale cohomology. \par
Given an abelian group $A$, we denote by $\prescript{}{m}{A}$ the $m$-torsion subgroup of $A$. The notation $A_\wedge$ is the inverse limit of the quotients $A/nA$ for all $n>0$.  For $A$ a topological abelian group, denote by $A^D:=\Hom_{\cts}(A,\QZ)$ the group of continuous homomorphisms from $A$ to $\QZ$. When there is no other topology defined on $A$, we equip $A$ with the discrete topology. The functor $A\mapsto A^D$ is an anti-equivalence of categories between torsion abelian groups and profinite groups.

\section{Failure of strong approximation}\label{failure}
Over $p$-adic function fields, we construct varieties of the form $\SL_n/H$ with $H$ semisimple simply connected and we show that strong approximation away from a finite set of places fails for such varieties. 
\par
\subsection{Inner type}\label{inner type}
For a smooth proper geometrically integral curve $X$ over a $p$-adic field $k$, we have the following dualities due to Lichtenbaum:
\begin{thm}[Lichtenbaum, cf. \cite{lichtenbaum1969duality} Theorem 4]
There are pairings \begin{equation}\label{pair}\psi:\br X\times \Pic X\rightarrow \BBQ/\BBZ
\end{equation}
\[\rho:H^1(k,\Pic \overline X)\times\Pic^0(X)\rightarrow \QZ\] where $\psi$ is induced by evaluation on closed points. These pairings induce dualities 
\begin{equation}\label{dual}\psi^* :\br X\simeq\Pic X^D\end{equation}
\begin{equation}\label{pic}\rho^* :H^1(k,\Pic(\overline X))\simeq\Pic^0(\overline X)^D.\end{equation}
\end{thm}

\begin{rmk}\label{rmk}
If $[A]$ comes from $\br k$ via the restriction map, then
$\psi([A],z)=(\deg z)[A].$
\end{rmk}

\begin{construction}\label{cons}Let $H$ be an absolutely simple simply connected group over $K$ of inner type $\mathsf A$. Then $H=\SL_1(A)$ for $A$ a simple central $K$-algebra. Suppose the class $[A]\in\br K$ lies in $\br_\nr (K/k)=\br X$.  
Embed $H$ into a special rational group $E$ which is split semisimple simply connected (for example, $E=\SL_n$ or $\Sp_{2n}$). Consider the classifying variety $Z:=E/H$. Since $H$ also embeds into the special rational group $\GL_1(A)$, we have that $Z$ is stably birational to the classifying variety $\GL_1(A)/\SL_1(A)=\gm$ which is rational. Therefore, the variety $Z$ is stably rational. In particular, weak approximation holds for $Z$.
\end{construction}
The Rost invariant $R_H$ has an explicit formula in this situation (cf. \cite{gille2011formules}): Let $[x]\in H^1(F,H)=F^\times/\Nrd((A\otimes F)^\times)$, then we have 
\[ R_H([x])=(x)\cup[A]\in H^3(F,\QZ(2))\]
where $(x)$ denotes the class in $H^1(K,\mu_m)=F^\times/F^{\times m}$ of any representative of $[x]$, and $m$ is the exponent of $A$. Moreover, the order of $R_H\in \Inv^3(H,\QZ(2))_\rmnorm$ equals $m$ (cf. Theorem 11.5 of \cite{garibaldi2003cohomological}).
\par 

We have the following commutative diagram where the first row is the complex (\ref{Weil}) by the generalized Weil reciprocity:
\begin{equation}\label{diagram}
\begin{tikzcd}
{H^3(K,\BBQ/\BBZ(2))} \arrow[r] & {\bigoplus_{v\in X^{(1)}}H^3(K_v,\BBQ/\BBZ(2))} \arrow[r,"\Sigma"] & \BBQ/\BBZ. \\
{H^1(K,H)} \arrow[r] \arrow[u,"R_H"]  & {\bigoplus_{v\in X^{(1)}} H^1(K_v,H)} \arrow[u,"\bigoplus R_{H}"] \arrow[ru,"f_H"']        &          
\end{tikzcd}\end{equation}
Now we show that the composed map $f_H$ defined by this diagram is related to the pairing $\psi$ in (\ref{pair}):
\begin{prop}\label{comp}
Let $a=([a_v])_v\in\bigoplus_{v\in X^{(1)}}H^1(K_v, H)=\bigoplus_{v\in X^{(1)}} K_v^\times/\Nrd((A\otimes K_v)^\times)$ represented by $a_v\in K_v^{\times}$. Denote by $n_v$ the valuation of $a_v$ with respect to $v$ and let $z:=\sum_{v\in X^{(1)}}n_v\cdot v \in\Pic X$. Then
$$\psi([A],z)=f_H(a).$$
\end{prop}

\begin{proof}We have $f_H(a)=\sum_{v\in X^{(1)}}\Cor_{\kappa(v)/k}(\partial_v(R_{H}([a_v])))$. By linearity, it suffices to show $$\psi([A],n_v\cdot v)=\Cor_{\kappa(v)/k}(\partial_v(R_{H}([a_v]))).$$ Choose a uniformizer $\pi_v$ of the place $v$, and write $a_v=u\pi_v^{n_v}$ with $u$ a unit in $\mathcal O_v$. Then $$\partial_v(R_{H}([a_v]))=\partial_v((u\pi_v^{n_v})\cup[A])=\partial_v((u)\cup[A])+n_v\partial_v((\pi_v)\cup[A]).$$ In fact $\partial_v((\pi_v)\cup [A])$ is the image of $[A]$ under the specialization map $$s_{-\pi_v}^2: H^2(K_v,\mu_{m})\rightarrow H^2(\kappa(v),\mu_{m})$$ with respect to the uniformizer $-\pi_v$ (cf. Construction 6.8.6 of \cite{gille2017central}). Actually $s_{-\pi_v}^2([A])$ is independent of the choice of the uniformizer $\pi_v$, since $\partial_v([A])=0$, implying that $[A]$ lies in the image of $H^2(\kappa(v),\mu_{m})\xrightarrow{\Inf}H^2(K_v,\mu_{m})$ (cf. Proposition 6.8.7 and Corollary 6.8.8 of \cite{gille2017central}). The fact that $[A]$ is unramified at $v$ also gives $\partial_v((u)\cup[A])=0$ (cf. Lemma 6.8.4 of \cite{gille2017central}). But $\psi([A],v)$ is just the evaluation of $[A]$ at the closed point $v$ followed by $\Cor_{\kappa(v)/k}$, so we have $$\psi([A],n_v\cdot v)=\Cor_{\kappa(v)/k}(n_v\partial_v((\pi_v)\cup[A]))=\Cor_{\kappa(v)/k}(\partial_v(R_H([a_v])))$$ giving the desired result.
\end{proof}
The above calculations also give the following result, which will be used later.
\begin{prop}\label{cyclic}There is a bijection
\begin{equation*}\begin{split}H^1(K_v,H)& \simeq \BBZ/m_v\BBZ\\
[a_v]&\mapsto\text{valuation of }a_v\bmod m_v
\end{split}
\end{equation*}
where $m_v$ denotes the exponent of $A\otimes K_v\in\br K_v$ and $a_v\in K_v^\times/K_v^{\times m_v}$ represents the class $[a_v]\in H^1(K_v,H)$. 
\end{prop}
\begin{proof}
The map given by the Rost invariant $$R_H:H^1(K_v,H)\rightarrow H^3(K_v,\QZ(2))$$ is injective in this case (cf. Theorem 5.3 of \cite{colliot2012patching}). Therefore, we can identify $H^1(K_v,H)$ with its image $R_H(H^1(K_v,H))$. Applying the isomorphism $\partial_v: H^3(K_v,\QZ(2))\simeq \br(\kappa(v))$, we see that $H^1(K_v,H)\simeq \partial_v(R_H(H^1(K_v,H)))$ is in fact a cyclic group generated by $\partial_v(R_H([\pi_v]))=v^*([A])\in\br(\kappa(v))$ with the calculations in Proposition \ref{comp}, where $v^*([A])$ denotes the evaluation of $[A]$ at $v$. By Proposition 6.8.7 and Corollary 6.8.8 of \cite{gille2017central}, the class $A\otimes K_v\in \br K_v$ comes from $v^*([A])$ under the injection $\br(\kappa(v))\hookrightarrow \br K_v$, and thus they have the same order.
\end{proof}

\par 
In virtue of the Lichtenbaum duality (\ref{dual}), the non-vanishing of the pairing $\psi$ gives non-vanishing $f_H$, which can then be used to fabricate examples of non-surjective $H^1(K,H)\rightarrow \bigoplus_{v\in X^{(1)}}H^1(K_v,H_v)$. Moreover, this map can be non-surjective even away from a finite set $S$ of places:
 
\begin{thm}\label{thm failure of sa}
If the set of values taken by $\psi([A],-)$ at
 the divisors supported in $S\subseteq X^{(1)}$ do not cover all the values taken at the divisors supported in $X^{(1)}\setminus S$, i.e.$$\{\psi([A],z)|\Supp(z)\subseteq  X^{(1)}\backslash S\}\not\subseteq  \{\psi([A],z)|\Supp(z)\subseteq S\},$$ then $$H^1(K,H)\rightarrow\bigoplus_{v\in X^{(1)}\backslash S} H^1(K_v,H)$$ is not surjective. As a result, strong approximation away from $S$ does not hold for $Z$.
\end{thm}
\begin{proof}
Take a divisor $z_0=n_v\cdot v$ supported in $X^{(1)}\backslash S$ such that $$\psi([A],z_0)\not\in \{\psi([A],z)|\Supp(z)\subseteq S\}.$$ Let $a:=([\pi_v^{n_v}])_v\in\bigoplus_{v\in X^{(1)}\backslash S}H^1(K_v,H)$. By Proposition \ref{comp}, the images of elements in $\bigoplus_{v\in S}H^1(K_v,H)$ under $f_H$ lie in $\{\psi([A],z)|\Supp(z)\subseteq S\}$, and thus cannot cancel out $f_H(a)=\psi([A],z_0)$. Then chasing the diagram (\ref{diagram}) shows that $a$ is not in the image of $H^1(K,H)$. 
\par Now consider the following commutative diagram
\[
\begin{tikzcd}
Z(K) \arrow[d] \arrow[r] & {H^1(K,H)} \arrow[d] \arrow[r]                 & 1 \\       Z(\BA^S_K) \arrow[r]        & {\bigoplus_{v\in X^{(1)}\backslash S}H^1(K_v,H)} \arrow[r] & 1
\end{tikzcd}\]
with exact rows. We identify $\oplus_{v\in X^{(1)}\backslash S}H^1(K_v,H)$ with the topological restricted product $\BP^1_S(H)$ because $H^1(\mathcal O_v, \mathcal H)=H^1(\kappa(v),H)=0$ by Serre's Conjecture II proved for $p$-adic fields (cf. Theorem 1 of \S4.1, \cite{kneser1969lectures}). The map $Z(\BA^S_K)\rightarrow\bigoplus_{v\in X^{(1)}\backslash S}H^1(K_v,H)$ is locally constant. Take one element $a\in\bigoplus_{v\in X^{(1)}\backslash S}H^1(K_v,H)$ which is not in the image of $H^1(K,H)$, for example the $a$ we just constructed. There exists an open subset $V$ of $z\in Z(\BA^S_K)$ which maps to $a$. A diagram chasing shows that $V$ doesn't meet $Z(K)$.
\end{proof}
For $S \subseteq X^{(1)}$, the index $I(S)$ of $S$ is defined to be the greatest common divisor of the degrees of all $v\in S$. The index $I(X)$ of $X$ is $I(S)$ for $S=X^{(1)}$.
 \begin{cor}\label{cor failure SA SLn/SL1A A constant}Let $A$ be a simple central $k$-algebra of exponent $m$ and consider $H:=\SL_1(A_K).$ Let $z_0=\sum_{v\in X^{(1)}} n_v\cdot v$ be a divisor such that $\deg z_0\not\equiv 0 \bmod m$. Let $S$ be a set of places in $X^{(1)}$ not meeting the support of $z_0$ such that $\deg(z_0)$ is not contained in the subgroup generated by $I(S)$ in $\BBZ/{m}\BBZ$. Then strong approximation away from $S$ does not hold for $Z$. 
\end{cor}
\begin{proof}
 By Remark \ref{rmk}, we have $\psi([A],z_0)=\deg z_0/{m}\in\prescript{}{m}\br k\simeq\frac{1}{m}\BBZ/\BBZ\subseteq\QZ$. When $z$ is supported in $S$, the value $\psi([A],z)$ is always a multiple of $I(S)/m$, which cannot be equal to $\deg(z_0)/m.$ The condition in Proposition \ref{thm failure of sa} is thus satisfied.
\end{proof}

\begin{eg}\label{example}
Let $X=\P_k^1$ be the projective line over $k=\Qp$. Then the function field $K=\Qp(t)$. Consider the quaternion algebra $A=(p,u)$ over $K$ with $u$ a non-square unit in $\Zp^\times$. Then $[A]\in\prescript{}{2}\br K$ comes from the generator of $\prescript{}{2}\br X=\prescript{}{2}\br k.$ Let $H$ be the semisimple simply connected group $\SL_1(A)$, and consider the stably rational homogeneous space $Z=\SL_n/H$. If $S\subseteq X^{(1)}$ is a subset containing only places of even degrees, for example $(t^2-p),(t^4-p),(t^6-p)$ etc, then strong approximation away from $S$ does not hold for $Z$. In fact, we will see later that strong approximation away from $S$ holds for $Z$ if and if only the index of $S$ is odd, and the reciprocity obstruction is the only one (cf. Theorem \ref{thm rec only one sc}).\end{eg}
\begin{rmk}
Our example shows that we have a different behavior over $p$-adic function fields compared to the classical situations over number fields. For $H$ a semisimple and connected algebraic group defined over a number field $F$, the canonical map $H^1(F,H)\rightarrow \prod_{v\in \Omega_F}H^1(F_v,H)$ is always surjective (cf. \S5.3 of \cite{kneser1969lectures}). The group $\SL_n$ satisfies strong approximation away from a non-empty set $S\subseteq \Omega_F$ of places. A diagram chasing argument shows that strong approximation away from a non-empty $S\subseteq \Omega_F$ always holds for $\SL_n/H$. In our case over a $p$-adic function field $K$, we still have strong approximation away from a non-empty $S\subseteq X^{(1)}$ for $\SL_n$ (cf. Corollary \ref{split groups satisfy SA}), but not for $\SL_n/H$ where $H$ is an absolutely simple simply connected algebraic group. We will give another example with $H$ of outer type $\mathsf A$ (cf. Example \ref{example outer type}).
\end{rmk}

\subsection{Outer type}

Now consider $H$ a simply connected group of outer type $\mathsf A_{n-1}$, i.e. $H=\SU(A,\tau)$, where $A$ is a central simple algebra of degree $n\geq 3$ over a (separable) quadratic extension $L$ of $K$ with a unitary involution $\tau$ which leaves $K$ element-wise invariant. The existence of such a $L/K$-unitary involution on $A$ is equivalent to $\Cor_{L/K}([A])=0\in\br K$ (cf. Theorem 3.1 of \cite{knus1998book}). Let $$\ssym(A,\tau)^\times=\{(s,z)\in A^\times\times L^
\times|\tau(s)=s,\Nrd_A(s)=N_{L/K}(z)\}$$ which is a homogeneous space for $\GL_1(A)$ under the operation given by $x:(s,z)\mapsto(xs\tau(x),z\Nrd_A(x))$. There is an exact sequence $1\rightarrow H\rightarrow \GL_1(A)\rightarrow \ssym(A,\tau)^\times\rightarrow 1$ which then induces the exact sequence $$\GL_1(A_{F\otimes L})\rightarrow \ssym(A_{F\otimes L},\tau)^\times\rightarrow H^1(F,H)\rightarrow H^1(F,\GL_1(A))=\{1\}$$ for any field extensions $F/K$. In other words, there is a canonical bijection between $H^1(F,H)$ and $\ssym(A_{F\otimes L},\tau)^\times/\approx$ where the equivalence relation $\approx$ is defined by $(s,z)\approx(s^\prime,z^\prime)$ if and only of $s^\prime=xs\tau(x)$ and $z^\prime=z\Nrd_{A_{F\otimes L}}(x)$ for some $x\in A_{F\otimes L}^\times$ (cf. (29.18) of \cite{knus1998book} and \S 5.5 of \cite{kneser1969lectures}). 
\par Over the quadratic extension $L/K$, the group  $H$ is isomorphic to $\SL_1(A)$. Under the field extension map, the Rost invariant $R_H$ maps to the Rost invariant $R_{H_L}$. The corestriction map for the field extension $L/K$ takes $R_{H_L}$ to $2R_H$. Using the formula of the Rost invariant for $\SL_1(A)$, we have
\[2R_H(s,z)=\Cor_{F\otimes L/F}((z)\cup[A_{F\otimes L}])\in H^3(F,\QZ(2))\]
over a field extension $F/K$, where $(s,z)\in\ssym(A_{F\otimes L},\tau)^\times/\approx.$

\par
We investigate different places $v\in X^{(1)}$ in terms of their ramification types. If $L_v:=K_v\otimes L$ is a field, then the valuation of $v$ on $K_v$ extends uniquely to $L_v$ (which we denote by $\tilde v$); moreover, if the residue field $\kappa(\tilde v)$ of $L_v$ is a field extension of degree $2$ (resp. $1$) of the residue field $\kappa(v)$ of $K_v$, we call such a place \textit{inert} (resp. \textit{ramified}). For an inert place $v$, the field extension $L_v/K_v$ is unramified. If $L_v$ is not a field, then we have $L_v\simeq K_v\times K_v$, and such a place is called \textit{totally split}.
\par
If $v$ is totally split, then $L$ embeds into $K_v$. The group $H_{K_v}$ is of inner type, and equals $\SL_1(A_{K_v}).$ We can thus apply our results in \S\ref{inner type}. In particular, the order of the Rost invariant $R_{H_{K_v}}$ equals the exponent $m_v$ of $A_{K_v}$. If $[A_{K_v}]$ is unramified at $v$, we can show that there is a bijection $H^1(K_v,H)\simeq \Z/m_v\Z$ as in Proposition \ref{cyclic}.
\begin{prop}
For an inert place $v$ such that $[A_{L_v}]$ is unramified at $\tilde v$, the image of $2R_H:H^1(K_v,H)\rightarrow H^3(K_v,\QZ(2))$ is $0$. 
\end{prop}
\begin{proof}
We use the diagram in Proposition 8.6 of \cite{garibaldi2003cohomological} which is commutative with exact rows:
\[\begin{tikzcd}
1\arrow[r]& H^i(\kappa(\tilde v),\mu_m^{\otimes j})\arrow[d,"e\cdot\Cor_{\kappa(\tilde v)/\kappa( v)}"]\arrow[r]&H^i(L_v,\mu_m^{\otimes j})\arrow[d,"\Cor_{L_v/K_v}"]\arrow[r,"\partial_{\tilde v}"]& H^{i-1}(\kappa(\tilde v),\mu_m^{\otimes(j-1)})\arrow[d,"\Cor_{\kappa(\tilde v)/\kappa( v)}"]\arrow[r]&1\\
1\arrow[r]& H^i(\kappa(v),\mu_m^{\otimes j})\arrow[r]&H^i(K_v,\mu_m^{\otimes j})\arrow[r, "\partial_{v}"]&H^{i-1}(\kappa(v),\mu_m^{\otimes(j-1)})\arrow[r]&1
\end{tikzcd}\addtag\label{inf residue split exact diagram}\]
where the ramification index $e=1$ in our case.
Since $[A_{L_v}]$ is unramified at $\tilde v$, it comes from a certain $\bar\alpha\in\br(\kappa(\tilde v))$ in the top row of the diagram. The condition $\Cor_{L_v/K_v}([A_{L_v}])=0$ then gives $\Cor_{\kappa(\tilde v)/\kappa (v)}(\bar\alpha)=0$ by taking $i=2$ and $j=1$ in the left square of the diagram. Now we compose $2R_H$ with the isomorphism $\partial_{v}: H^3(K_v,\QZ(2))\xrightarrow{\simeq}\br\kappa(v)$ and we get 
\[\partial_v(\Cor_{L_v/K_v}((z)\cup [A_{L_v}]))=\Cor_{\kappa(\tilde v)/\kappa (v)}(\partial_{\tilde v}((z)\cup [A_{L_v}] ))=\Cor_{\kappa(\tilde v)/\kappa (v)}(\bar\alpha^{\tilde v(z)})=0\]
by taking $i=3$ and $j=2$ in the right square of the diagram, where we denote by $\tilde{v}(z)$ the valuation of $z$ with respect to $\tilde v$.
\end{proof}
\begin{cor}\label{cor outer type}
Suppose there is a totally split place $v_0$ such that $[A_{K_{v_0}}]$ is unramified at $v_0$ with exponent $\geq 3$. Let $S\subseteq X^{(1)}$ be a finite set containing inert places $v$ such that $[A_{L_v}]$ is unramified at $\tilde {v}$. Then $$H^1(K,H)\rightarrow\bigoplus_{v\in X^{(1)}\backslash S} H^1(K_v,H)$$ is not surjective. As a result, strong approximation away from $S$ does not hold for $Z=E/H$ which is a classifying variety of $H$. 
\end{cor}
\begin{proof}
Proposition \ref{cyclic} applies to $H^1(K_{v_0},H)=H^1(K_{v_0},\SL_1(A))$ and we know that it is in fact a cyclic group of order equal to the exponent of $[A_{K_{v_0}}]$. Construct $$(a_v)_v\in\oplus_{v\in X^{(1)}\backslash S}H^1(K_v, H)$$ such that $$a_{v_0}=[\pi_{v_0}]\in K_{v_0}^\times/\Nrd(A_{K_{v_0}})\simeq H^1(K_{v_0},H)$$ for a uniformizer $\pi_{v_0}$ of $K_{v_0}$ and $a_v=0$ elsewhere. Then apply Proposition \ref{comp} and we get $f_H(a)$ is of order $\geq 3$. For $v\in S$, the images of $H^1(K_v,H)$ under $f_H$ only give elements of order at most $2$, and thus cannot cancel out $f_H(a)$. The same diagram chasing argument as in Theorem \ref{thm failure of sa} then gives the result.
\end{proof}
There are only finitely many places left: those $v$ such that the field extension $L_v/K_v$ is ramified, or $[A_{L_v}]$ is ramified at $\tilde v$. In fact, these two conditions cannot hold at the same time. \par If $L_v/K_v$ is ramified, then by Lemma 6.3 of \cite{parimala2022local}, we have $[A_{L_v}]=\alpha_0 \otimes L_v$ for some unramified $\alpha_0\in\br_\nr(K_v/K).$ The compatibility between residue maps and restriction maps (cf. Proposition 1.4.6 of \cite{colliot2021brauer}) shows that $[A_{L_v}]$ is unramified at $\tilde v$. Moreover, the exponent of $[A_{L_v}]$ is at most $2$ because \[0=\Cor_{L_v/K_v}([A_{L_v}])=\Cor_{L_v/K_v}(\Res_{L_v/K_v}(\alpha_0))=2\alpha_0.\]

\begin{lem}\label{lemma structure of h1 h2 kv}
For $v\in X^{(1)}$, we have $$H^1(K_v,\mu_m)\simeq K_v^\times/ K_v^{\times m}=\{\pi^{r}\delta^{s}u^{t}|(r,s,t)\in(\Z/m\Z)^3\}\simeq (\Z/m\Z)^3$$
where $\pi$ is a uniformizer for $K_v$, the element $\delta\in K_v$ is a unit whose image $\bar \delta\in\kappa (v)$ is a uniformizer for $\kappa(v)$, and $u$ is a unit in $K_v$ whose image $\bar u\in\kappa(v)$ is a unit of order $m$ in $\kappa(v^\times)/\kappa(v)^{\times m}$.\par 
If $K_v$ contains a primitive $m$th root of unity $\omega$ (e.g. when $m|(p-1)$), then we have
\[H^2(K_v,\mu_m)\simeq \prescript{}{m}{\br(K_v)}=\{r(\delta,u)_\omega+s(\pi,\delta)_\omega+t(\pi,u)_\omega|(r,s,t)\in(\Z/m\Z)^3\}\simeq (\Z/m\Z
)^3\]
where $(a,b)_\omega$ is a cyclic algebra of degree $m$.
\end{lem}
\begin{proof}
The exact rows in (\ref{inf residue split exact diagram}) are in fact split (cf. Corollary 6.8.8 of \cite{gille2017central}). Taking $i=1, j=1$ gives the first statement, noting that 
\[H^1(\kappa(v),\mu_m)\simeq \kappa(v)^\times/\kappa(v)^{\times m}=\{\bar \delta^s \bar u^t|(s,t)\in(\Z/m\Z)^2 \}\simeq (\Z/m\Z)^2\]
which is obtained by applying the same exact sequence to the complete discretely valued field $\kappa(v)$.
\par A primitive $m$th root of unity $\omega$ induces an isomorphism $\mu_m\simeq \Z/m\Z$. The group $H^2(\kappa(v),\mu_m)\simeq \prescript{}{m}{\br(\kappa(v))}\simeq\Z/m\Z$ is then generated by the cyclic algebra $(\bar\delta,\bar u)_{\bar\omega}$, where $\bar\omega$ denotes the image of $\omega$ in the residue field $\kappa(v)$. Take $i=2,j=1$ in the split exact sequence (\ref{inf residue split exact diagram}). A retract of the injection $\prescript{}{m}{\br\kappa(v)}\hookrightarrow \prescript{}{m}{\br K_v}$ is given by the specialization map $s^2_{\pi}:a\mapsto \partial_v((-\pi)\cup a)$. The calculations  
\[s_\pi^2(r(\delta,u)_\omega+s(\pi,\delta)_\omega+t(\pi,u)_\omega)=r(\bar\delta,\bar u)_{\bar\omega},\]
\[\partial_v(r(\delta,u)_\omega+s(\pi,\delta)_\omega+t(\pi,u)_\omega)=\bar\delta ^s\bar u ^t\]
give the second statement.
\end{proof}
If the field extension $L_v/K_v$ is unramified and $[A_{L_v}]$ has index $n\geq 3$, then by Proposition 6.6 of \cite{parimala2022local}, there is a primitive $n$th root of unity $\rho$ in $L_v$ such that $N_{L_v/K_v}(\rho)=1$ and $[A_{L_v}]=(\pi,\delta)_\rho$, where $\pi$ and $\delta$ are as in Lemma \ref{lemma structure of h1 h2 kv}. In particular, $[A_{L_v}]$ is ramified at $\tilde v$. There exists an $L_v/K_v$-unitary involution $\tau$ on $(\pi,\delta)_\rho$ such that $\tau$ leaves $i$ and $j$ invariant, where $i$ and $j$ are generators of $(\pi,\delta)_\rho$ such that $i^{n}=\pi,j^{n}=\delta,ij=\rho ji$.

\begin{eg}\label{example outer type}
Let $X=\BP_k^1$ be the projective line over $k$ with function field $K=k(t)$. Let $k^\prime=k[T]/g(T)$ be an unramified quadratic extension of $k$, such that there is a primitive $m$th root of unity $\rho$ with $N_{k^\prime/k}(\rho)=1$ and $m\geq 3$. Examples include $k^\prime =k[T]/(T^2+T+1)$ over $k=\BBQ_5$ with $\rho$ a primitive $3$rd root of unity, and $k^\prime=k[T]/(T^2-\frac{1+\sqrt 5}{2}T+1)$ over $k=\BBQ_3(\sqrt{-7})$ with $\rho$ a primitive $5$th root of unity, etc. Let $L=k^\prime (X_{k^{\prime}})=k^\prime(t)$ which is a quadratic extension of $K$. Let $\delta$ be a uniformizer of $k$, which stays a uniformizer of $k^\prime$ since the extension $k^\prime/k$ is supposed to be unramified. Let $u\in k^\prime$ be a unit with order $m$ in $k^{\prime\times}/k^{\prime\times m}$. Consider the cyclic $L$-algebra $A=(\delta,u)_{\rho}$. The non-trivial $L/K$-automorphism extends to an $L/K$-unitary involution $\tau$ on $A$. Let $H=\SU(A,\tau)$ and $Z=E/H$ be a classifying variety. The place $v_0=(g(t))$ is totally split, and $H^1(K_{v_0},H)$ is cyclic of order $m$. There are no ramified places or places $v$ where $[A_{L_k}]$ is ramified. Let $S$ be a subset of places $v$ where $L_v/K_v$ are (unramified) field extensions. Then strong approximation away from $S$ does not hold for $Z$.
\end{eg}

\section{Reciprocity obstruction to strong approximation}
We establish a reciprocity obstruction to strong approximation. This obstruction for classifying varieties is related to the degree $3$ cohomological invariants. As a result, the failure of strong approximation constructed in Section \ref{failure} is explained by this reciprocity obstruction. We prove that this reciprocity obstruction to strong approximation is the only one in certain situations, including the case of Example \ref{example}, and classifying varieties of tori. We also explain that for $K$-tori, this reciprocity obstruction to strong approximation is compatible with the known results in \cite{harari2019espace}.
\subsection{Defining the reciprocity obstruction}\label{section def rec obs}
Let $Z$ be a smooth geometrically integral $K$-variety. As a consequence of Gersten's conjecture for étale cohomology proved by Bloch and Ogus (cf. \cite{bloch1974gersten}), we have:
\begin{thm-df}[cf. Theorem 4.1.1 of \cite{colliot1995birational}, see also (1.1) of \cite{colliot2013descente}]\label{def}The following subgroups of $H^3(K(Z),\QZ(2))$ coincide, that we define to be $H^3_{\nr}(Z,\QZ(2))$:
\begin{enumerate}[label=(\arabic*)]
    \item the group $H_\Zar^0(Z,\mathcal H^3(\QZ(2)))$ of global sections of the Zariski sheaf $\mathcal H^3(Z,\QZ(2))$ which is the sheaf associated to the Zariski presheaf $U\mapsto H^3_\et(U,\QZ(2))$;
    \item the group of elements $\alpha\in H^3(K(Z),\QZ(2))$ that are unramified with respect to any codimension $1$ point $P$ of $Z$, i.e. we have $\partial_{\mathcal O_{Z,P}}(\alpha)=0\in H^2(\kappa(P),\QZ(1))$ or equivalently $\alpha$ comes from a class in $H^3(\mathcal O_{Z,P},\QZ(2))$;
    \item\label{3} the group of elements in $H^3(K(Z),\QZ(2))$ which at any point $P\in Z$ come from a class in $H^3(\mathcal O_{Z,P},\QZ(2))$.
\end{enumerate}
\end{thm-df}
\begin{rmk}
Since we do not suppose $Z$ to be proper, the group $H^3_\nr(Z,\QZ(2))$ can be strictly bigger than its subgroup $H^3_\nr(K(Z)/K,\QZ(2))$, which is the group of elements in $H^3(K(Z),\QZ(2))$ that are unramified with respect to all discrete valuations of $K(Z)$ trivial on $K$. Let $Z^c$ be a smooth compactification of $Z$, then we have $H^3_\nr(Z^c,\QZ(2))=H^3_\nr(K(Z)/K,\QZ(2))$. We will see later in Theorem \ref{thm rec only one sc} that for $Z$ in our Construction \ref{cons}, the group $H^3_\nr(Z,\QZ(2))/\Im(H^3(K,\QZ(2)))$ is of order $m$ which accounts for the defect of strong approximation for $Z$, but $$H^3_\nr(K(Z)/K,\QZ(2))/\Im(H^3(K,\QZ(2)))=0$$ since $Z$ is stably birational.
\end{rmk}

Now we define a pairing 
\begin{equation}\label{pairing adelic space H3nr}Z(\BA_K)\times H^3_\nr(Z,\QZ(2))\rightarrow\QZ\end{equation}
in a similar way as in \S2.2 of \cite{colliot2016lois}, where their original definition applies to $H^3_\nr(K(Z)/Z,\QZ(2))$ instead of $H^3_\nr(Z,\QZ(2))$ (see also Lemma 4.1 of \cite{harari2015weak}). Let $\alpha\in H^3_\nr(Z,\QZ(2))$. For any field extension $F/K$ and any point $x\in Z(F)$ with image $P\in Z$, we can lift $\alpha$ to a (unique) element of $H^3(\mathcal O_{Z,P},\QZ(2))$ in virtue of \ref{def}\ref{3}, and define $\alpha(x)$ to be its image under the pullback $H^3(\mathcal O_{Z,P},\QZ(2))\rightarrow H^3(F,\QZ(2))$. This defines a pairing 
\begin{align}\label{pairing at one place}Z(K_v)\times H^3_\nr(Z,\QZ(2))&\rightarrow H^3(K_v,\QZ(2))=\QZ\\
(x,\alpha)&\mapsto \alpha(x)\nonumber.\end{align}
By shrinking the open subset $U\subseteq X$ over which we have the model $\mathcal Z\rightarrow U$ of $Z$, the element $\alpha$ lies in $H^3(\mathcal O_{\mathcal Z,Q},\QZ(2))$ for all but finitely many codimension $1$ points $Q$ of the $k$-variety $\mathcal Z$. By the assumption on $\alpha$, the exceptional $Q$ lie in finitely many closed fibers $\mathcal Z_{v_1}\cdots,\mathcal Z_{v_r}$ of $\mathcal Z\rightarrow U$. Hence for $v\neq v_i$, we have $\alpha\in H^3(\mathcal O_{\mathcal Z,P},\QZ(2))$ for all points $P\in\mathcal Z_v$. Therefore, for $x$ coming from $\mathcal Z(\mathcal O_v)$, we have $\alpha(x)\in H^3(\mathcal O_v,\QZ(2))=H^3(\kappa(v),\QZ(2))=0$, noting that the residue field $\kappa(v)$ has cohomological dimension $2$. We can thus sum up these maps to get a well-defined pairing (\ref{pairing adelic space H3nr}) for the adelic space.
\begin{rmk}
In the particular case when $Z$ is proper (so strong approximation is equivalent to weak approximation), the above pairing (\ref{pairing adelic space H3nr}) becomes 
\[\prod_{v\in X^{(1)}}Z(K_v)\times H^3_{\nr}(K(Z)/K,\QZ(2))\rightarrow \QZ\]
which is exactly the pairing in \S4 of \cite{harari2015weak} giving the reciprocity obstruction to weak approximation. 
\end{rmk}
For a subset $S\subseteq X^{(1)}$, we define the subgroup of ``trivial on $S$'' elements:
\begin{equation*}H^3_{\nr,S}(Z,\QZ(2)):=\ker(H^3_\nr(Z,\QZ(2))
\rightarrow\prod_{v\in S}H^3_\nr(Z_{K_v},\QZ(2))).\end{equation*}
In virtue of the complex (\ref{Weil}) of the generalized Weil reciprocity law, the subset $Z(K)\subseteq Z(\BA_K^S)$ is orthogonal to $H^3_{\nr,S}(Z,\QZ(2))$ with respect to the pairing
\[Z(\BA_K^S)\times H^3_\nr(Z,\QZ(2))\rightarrow\QZ,\quad ((x_v)_{v\notin S},\alpha)\mapsto\sum_{v\notin S}\alpha(x_v).\]
  The same holds for the closure $\overline{Z(K)}^S$ of $Z(K)$ in $Z(\BA_K^S)$ by continuity of the pairing, 
  giving rise to a reciprocity obstruction to strong approximation away from $S$. Modulo constant elements, we can also replace $H^3_{\nr,S}(Z,\QZ(2))$ by
  \[\overline H^3_{\nr,S}(Z,\QZ(2)):= H^3_{\nr,S}(Z,\QZ(2))/\Im(\ker(H^3(K,\QZ(2))\rightarrow \prod_{v\in S}H^3(K_v,\QZ(2))))\]
which is a subgroup of $\overline H^3_{\nr}(Z,\QZ(2)):=H^3_{\nr}(Z,\QZ(2))/\Im(H^3(K,\QZ(2)))$. We say that \textit{the reciprocity obstruction to strong approximation away from $S$ is the only one} if the subset of elements in $Z(\BA_K^S)$ orthogonal to $H^3_{\nr,S}(Z,\QZ(2))$ (or equivalently $\overline H^3_{\nr,S}(Z,\QZ(2))$) equals $\overline{Z(K)}^S.$


\begin{rmk}

In the study of local-global principal for torsors under tori over $K$, Harari and Szamuely (cf. \cite{harari2013local}) used the reciprocity obstruction given by a different pairing $$Z(\BA_K)\times H^3(Z,\QZ(2))\rightarrow\QZ.$$ This pairing is compatible with our pairing (\ref{pairing adelic space H3nr}) via a map $H^3(Z,\QZ(2))\rightarrow H^3_\nr(Z,\QZ(2))$ that we describe now.  By results of Bloch and Ogus (\cite{bloch1974gersten}), the spectral sequence \[E_2^{p,q}=\hz^p(Z,\CH^q(\mu_m^{\otimes j}))\implies H^{p+q}_\et(Z,\mu_m^{\otimes j})\] gives rise to an exact sequence
\[H^3(Z,\mu_m^{\otimes j})\rightarrow H^0(Z,\CH^3(\mu_m^{\otimes j}))
\rightarrow H^2(Z,\CH^2(\mu_m^{\otimes j}))\rightarrow H^4(Z,\mu_m^{\otimes 2})\]
which for $j=2$ can be rewritten as
\[H^3(Z,\mu_m^{\otimes 2})\rightarrow H^3_\nr(Z,\mu_m^{\otimes 2})
\rightarrow\CHOW^2(Z)/m\rightarrow H^4(Z,\mu_m^{\otimes 2})\]
where $\CHOW^2(Z)$ denotes the second Chow group of codimension $2$ cycles modulo rational equivalence. When $\CHOW^2(Z)/m$ vanishes, the obstructions given by $H^3(Z,\mu_m^{\otimes 2})$ and $H^3_\nr(Z,\mu_m^{\otimes 2})$ are the same. In general, the obstruction given by $H^3_\nr(Z,\QZ(2))$ is potentially more restrictive than the obstruction given by $H^3_\nr(Z,\QZ(2))$. 

\end{rmk}
\begin{rmk}\label{rmk new def of obstruction}
In the literature of strong approximation problems away from $S$ over a number field $F$, the Brauer-Manin obstruction used to be defined using the projection $\pi_S:Z(\BA_F)\rightarrow Z(\BA_F^S)$, saying that the Brauer-Manin obstruction to strong approximation away from $S$ is the only one if $Z(F)$ is dense in $\pi_S(Z(\BA)^{\br})$. Demeio introduced the modified Brauer group of ``trivial on $S$'' elements (cf. \S6 of \cite{demeio2022etale})
\[\br_S Z:=\Ker(\br Z\rightarrow\prod_{v\in S}\br Z_{F_v})\]
and the corresponding Brauer set $Z(\BA_F^S)^{\br_S}$ (for which we can also replace $\br_S Z$ by $\br_S Z/\Im(\br_S F)$). This gives a less restrictive obstruction compared to the one defined using projection, since $\pi_S(Z(\BA_F)^{\br})\subseteq Z(\BA_F^S)^{\br_S}$; moreover, we have $\overline{Z(F)}^S\subseteq Z(\BA_F^S)^{\br_S}$ while $\pi_S(Z(\BA_F)^{\br})$ can be strictly smaller than $\overline{Z(F)}^S$ (cf. Proposition 4.8 of \cite{demeio2022etale}). We adopt this point of view and we will show later (cf. Theorem \ref{thm rec only one sc} and \ref{thm rec only one for SLn/tori}) that in our case, the group $\overline H^3_{\nr,S}(Z,\QZ(2))$ of ``trivial on $S$'' elements cuts out exactly the closure $\overline{Z(K)}^S$ with respect to the reciprocity obstruction.
\end{rmk}

\subsection{Application to classifying varieties}
Now we calculate this reciprocity obstruction we just defined for classifying varieties. Let $H$ be an algebraic group over $K$. Let $Z=E/H$ be a classifying variety where the special rational group $E$ is split semisimple simply connected (for example, $E=\SL_n$ or $\Sp_{2n}$).

The following morphism is defined by evaluating an invariant $I\in \Inv^3(H,\QZ(2))$ at the generic $H$-torsor $\xi^*E$ (the fiber of the $H$-torsor $E\rightarrow Z$ above the generic point $\xi\in Z$): 
$$\Inv^3(H,\QZ(2))\rightarrow H^3(K(Z),\QZ(2)).$$ 
Rost proved that this map is injective, and more precisely, there is an isomorphism
\begin{equation}\label{theta}\theta:\Inv^3(H,\QZ(2))\simeq H^3_\nr(Z,\QZ(2))\end{equation} (cf. Part 2, Theorem 3.3 and Part 1, Appendix C of \cite{garibaldi2003cohomological}). The decomposition $$\Inv^3(H,\QZ(2))=\Inv^3(H,\QZ(2))_\rmnorm\oplus H^3(K,\QZ(2))$$ then induces an isomorphism 
\[\begin{split}H^3_{\nr,S}(Z,\QZ(2))\simeq \ker(\Inv^3(H,\QZ(2))_\rmnorm\rightarrow\prod_{v\in S} \Inv^3(H_{K_v},\QZ(2))_\rmnorm)\\
\oplus \ker(H^3(K,\QZ(2))\rightarrow \prod_{v\in S}H^3(K_v,\QZ(2)))\end{split}\]
and modulo the constant elements,
 we get 
 \[\overline H^3_{\nr,S}(Z,\QZ(2))\simeq \ker(\Inv^3(H,\QZ(2))_\rmnorm\rightarrow\prod_{v\in S} \Inv^3(H_{K_v},\QZ(2))_\rmnorm).\addtag\label{theta mod constant}\]

\begin{prop}\label{prop inv rec compatible}
There is a well-defined pairing
\begin{align}\label{inv pairing formula}Z(\BA_K^S)\times\Inv^3(H,\QZ(2))_\rmnorm&\rightarrow \QZ \\
((x_v)_v,I) &\mapsto \sum_{v\in X^{(1)}\backslash S}\Cor_{\kappa(v)/k}(\partial_v(I_{K_v}(x_v^*E)))\nonumber\end{align}
induced by evaluating an invariant $I$ at the $H$-torsors $x_v^*E$ over $K_v$, where $x_v^*E$ denotes the fiber of the $H$-torsor $E\rightarrow Z$ above the point $x_v\in Z(K_v)$. 
This pairing (\ref{inv pairing formula}) is compatible with (\ref{pairing adelic space H3nr}) in the sense of the commutative diagram:
\[
\begin{tikzcd}
Z(\BA_K^S) \arrow[d,Rightarrow, no head] \arrow[r,phantom,"\times" description] & {\Inv^3(H,\QZ(2))_\rmnorm} \arrow[r] \arrow[d, "\simeq","\theta"' ]       & \QZ \arrow[d,Rightarrow, no head] \\
Z(\BA_K^S)                                \arrow[r,phantom,"\times" description] & {\overline H^3_\nr(Z,\QZ(2))} \arrow[r]  & \QZ.                               
\end{tikzcd}
\]
\end{prop}
\begin{proof}
Let $b\in H^3_\nr(Z,\QZ(2))$ and $x\in Z(F)$ be a point over any field extension $F/K$. We use the description of the isomorphism $\theta$ in Appendix A.2, p.461 of \cite{merkurjev2002unramified}. We define the class $b(x)\in H^3(F,\QZ(2))$ as the image of $b$ under the pull-back morphism 
\[x^*:H_\Zar^0(Z,\mathcal H^3(\QZ(2)))\rightarrow H_\Zar^0(\Spec F, \mathcal H^3(\QZ(2)))=H^3(F,\QZ(2))\] with respect to $x:\Spec F\rightarrow Z$. Thus we get a map
\[\hat I_K: Z(K)\rightarrow H^3(L,\QZ(2)),\quad x\mapsto b(x)\]
which is exactly how we defined our pairing (\ref{pairing at one place}). This $\hat I_K$ defines an invariant $I\in\Inv^3(H,\QZ(2))$ with $\theta(I)=b$ since the map $b$ is constant on the orbits of the $E(K)$-action on $Z(E)$. Therefore, we have a commutative diagram \[\begin{tikzcd}
Z(K_v) \arrow[d,Rightarrow, no head] \arrow[r,phantom,"\times" description] & {\Inv^3(H,\QZ(2))} \arrow[r]        & \QZ \arrow[d,Rightarrow, no head] \\
Z(K_v)                              \arrow[r,phantom,"\times" description] & {H^3_\nr(Z,\QZ(2))} \arrow[r] \arrow[u, leftarrow,"\simeq","\theta"' ] & \QZ.                            
\end{tikzcd}\]
Now we sum up these maps and quotient by $H^3(K,\QZ(2))$ and we get the desired result.
\end{proof} 
As a result, there is an reciprocity approximation to strong approximation which accounts for our constructions in Section \ref{failure}. Now we study more precise questions of how this reciprocity obstruction controls strong approximation, for example, is this obstruction the only one? For this end, we first need to take a closer took at our $p$-adic function field $K$.
\begin{prop}\label{exponent relation brk br K}
Let $X$ be a smooth proper geometrically integral curve over a $p$-adic field $k$, with index $I(X)$ and function field $K.$ Let $[A_0]\in \br k$ of exponent $m_0$ which maps to $[A]\in\br K$ of exponent $m$. Then $$m=\frac{m_0}{\gcd(I(X),m_0)},$$
and we can always choose $[A_1]\in\br k$ of exponent $mI(x)$ which maps to $[A]\in\br K$.
\end{prop}
\begin{proof}
Since $[A_{K_v}]\in \br K_v$ is unramified, it comes from its evaluation $v^*[A_{K_v}]\in \br \kappa(v)$ at the place $v$, which is also the image of $[A_0]$ under the restriction map $\br k\rightarrow
\br(\kappa(v))$, given by multiplication by $[\kappa(v):k]=\deg v$ if we identify both $\br k$ and $\br(\kappa(v))$ with $\QZ$. Therefore, the exponent $m_v$ of $v^*[A_{K_v}]$ equals $\frac{m_0}{\gcd(\deg v,m_0)}$, which is also the exponent of $A_{K_v}$ because of the injection $\br\kappa(v)\hookrightarrow \br K_v$. Each exponent $m_v$ of $A_{K_v}$ divides the exponent $m$ of $A$, and thus their least common multiple $\lcm(m_v)_{v\in X^{(1)}}$ divides $m$. On the other hand $\lcm(m_v)_{v\in X^{(1)}}$ annihilates all the $v^*[A_{K_v}]$, and thus $\lcm(m_v)_{v\in X^{(1)}}[A]\in \br X$ vanishes at all the closed points $v\in X^{(1)}$. By Lichtenbaum's duality (\ref{dual}), this implies that $\lcm(m_v)_{v\in X^{(1)}}[A]=0\in\br X$ and hence $m|\lcm(m_v)_{v\in X^{(1)}}$. Therefore, we have $m=\lcm(m_v)_{v\in X^{(1)}}.$ But
\[\lcm(m_v)_{v\in X^{(1)}}=\lcm(\frac{m_0}{\gcd(\deg v,m_0)})_{v\in X^{(1)}}=\frac{m_0}{\gcd(I(X),m_0)},\] giving the first statement.
\par In particular, taking $m=1$ shows that $\ker(\br k\rightarrow\br K)$ is the cyclic subgroup of order $I(X)$ in $\br k\simeq \QZ$. Choose $\alpha\in\Ker(\br k\rightarrow \br K)$ of order exactly $I(X)$. Then $[A_1]=[A_0]+\alpha\in \br k $ has exponent $$\lcm(m_0,I(X))=\frac{m_0 I(X)}{\gcd(I(X),m_0)}=mI(X)$$
and $[A_1]$ maps to $[A]\in\br K$ as desired.
\end{proof}
\begin{rmk}
In fact, the particular case $m=1$ showing that  $\ker(\br k\rightarrow\br K)$ has order $I(X)$ gives exactly Theorem 1 of \cite{roquette1966splitting} and Theorem 3 of \cite{lichtenbaum1969duality}, where in the latter paper Lichtenbaum reduced the calculations to the period of $X$ and used the duality (\ref{pic}).
\end{rmk}
\begin{rmk}\label{local-global for brauer}
In the proof of Proposition \ref{exponent relation brk br K}, we see that Lichtenbaum's duality (\ref{dual}) can actually give a local-global principle for $\br K$, i.e. the natural map
\[\br K\rightarrow\prod_{v\in X^{(1)}}\br K_v\]
is injective. If $\alpha\in\br K$ vanishes in $\br K_v$, then $\alpha$ is unramified because the residue map factors through the completion. Hence $\alpha$ lies in $\br X$ with evaluation $v^*(\alpha)=0\in\br\kappa(v)$ for all places $v$, and the duality (\ref{dual}) shows that $\alpha=0\in\br K$. See also Corollary 10.5.5 and Remark 10.5.7(2) of \cite{colliot2021brauer}.
\end{rmk}
\begin{thm}\label{thm rec only one sc}
Let $Z=E/\SL_1(A)$ as in Construction \ref{cons}, with $A$ of exponent $m$. For the $p$-adic curve $X$, suppose $\Pic^0(X)/m=0$, or equivalently $\prescript{}{m} H^1(k,\Pic\overline X)=0$ by the Lichtenbaum duality (\ref{pic}) (for example $X=\P^1_k$ the projective line satisfies this condition). Then there is an exact sequence of pointed sets
$$1\rightarrow \overline{Z(K)}^S\rightarrow Z(\BA_K^S)\rightarrow \overline H^3_{\nr,S}(Z,\QZ(2))^D
\rightarrow 1$$ for $S\subseteq X^{(1)}$ a non-empty finite set of places. In particular,
 the reciprocity obstruction to strong approximation away from $S$ is the only one for $Z$. 
 \par The group $\overline H^3_{\nr,S}(Z,\QZ(2))$ measures the defect of strong approximation away from $S$ for $Z$, and is finite cyclic of order $\gcd(\frac{I(S)}{I(X)},m)$ where $I(X)$ (resp. $I(S)$) is the index of $X$ (resp. $S$). In particular, strong approximation away from $S$ holds for $Z$ if and only of $I(S)/I(X)$ is coprime to $m$, and such an $S$ always exists, for example $S$ such that $I(S)=I(X)$.
\end{thm}
\begin{proof}We claim that there is a commutative diagram as follows, and the right column is an exact sequence:
\[
\begin{tikzcd}
E(K)\arrow[r]\arrow[d]&Z(K) \arrow[d] \arrow[r] & {H^1(K,H)} \arrow[d] \arrow[r]                 & 1 \\ E(\BA_K^S)      \arrow[r]&Z(\BA_K^S) \arrow[r]\arrow[d]        & {\bigoplus_{v\in X^{(1)}\backslash S}H^1(K_v,H)}\simeq \bigoplus_{v\in X^{(1)}\backslash S}\BBZ/m_v\BBZ \arrow[r]\arrow[d, "(n_v)_v\mapsto \Sigma n_v\cdot \deg v /I(X)", "f_H"'] & 1\\
&\overline H^3_{\nr,S}(Z,\QZ(2))^D\arrow[r,"h\mapsto h(\frac{m}{\gcd(\frac{I(S)}{I(X)},m)}R_H)", "\simeq"']&\BBZ/\gcd(\frac{I(S)}{I(X)},m)\BBZ\arrow[d]\\
&&1&
\end{tikzcd}\]
where $m_v=mI(X)/\gcd(\deg v,mI(X))$, and we identify ${\bigoplus_{v\in X^{(1)}\backslash S}H^1(K_v,H)}$ again with $\BP_S^1(H)$ as explained in Theorem \ref{thm failure of sa}.
\par
First we prove that it holds when we take $S=\emptyset$ and the corresponding $\gcd(\frac{I(S)}{I(X)},m)=m.$
\par\noindent\textbf{Definition and exactness of the right column.} The Hochschild-Serre spectral sequence $$H^p(k,H^q(\overline X,\gm))\implies H^{p+q}(X,\gm)$$ along with $\br(\overline X)=0$ gives the exact sequence 
\[0\rightarrow \Pic X\rightarrow H^0(k,\Pic \overline X)\rightarrow\br k\rightarrow \br X\rightarrow H^1(k,\Pic \overline X).\]
Since $\prescript{}{m}{H^1(k,\Pic\overline X)=0}$, the class $[A]\in\prescript{}{m}{\br X}$ comes from $\br k$. By Proposition \ref{exponent relation brk br K}, we can choose $[A_0]\in\br k$ of exponent $mI(X)$ which maps to $[A]\in\br K$. The exponent of $A_{K_v}$ equals $m_v=mI(X)/\gcd(\deg v,mI(X))$ as explained in the proof of Proposition \ref{exponent relation brk br K}. \par
Let $([a_v])_v\in \bigoplus_{v\in X^{(1)}}H^1(K_v,H)$ represented by $a_v\in K_v^\times$, and let $z_0=\sum_{v\in X^{(1)}}n_v\cdot v\in\Pic X$ where $n_v$ is the valuation of $a_v$ at $v$. Then $(n_v)_v\in \bigoplus_{v\in X^{(1)}}\BBZ/m_v\BBZ$ corresponds to $([a_v])_v\in \bigoplus_{v\in X^{(1)}}H^1(K_v,H)$ under the isomorphism of Proposition \ref{cyclic}. By Proposition \ref{comp} and Remark \ref{rmk}, the image of $f_H$ (the map $\Sigma\circ\oplus R_H$ in the diagram (\ref{diagram})) is contained in the subgroup $\prescript{}{mI(X)}{\br k}=\BBZ/mI(X)\BBZ$ generated by $[A_0]$, and
under this identification $[A_0]=1\in \BBZ/mI(X)\BBZ$, we have \[f_H(([a_v])_v)=\deg z_0\bmod mI(X)=\sum_{v\in X^{(1)}}(n_v\bmod m_v)\cdot\deg v\mod mI(X).\addtag\label{sum of degrees}\] Note that the sum on the right is indeed well-defined because $mI(X)|(m_v\cdot\deg v)$. The image of the degree map in $\BBZ/mI(X)\BBZ$ is the cyclic subgroup of order $m$ generated by $I(X)$. Hence dividing (\ref{sum of degrees}) by $I(X)$, we can identify the image of $f_H$ with $\BBZ/m\BBZ$, taking $I(X)[A_0]\in\br k$ to $1\in\BBZ/m\BBZ$ as indicated in the diagram. 
\par Now suppose $f_H(([a_v])_v)=0= \deg z_0/I(X)\bmod m,$ or equivalently $\deg z_0 [A_0]=0\in\prescript{}{mI(X)}{\br k}$. Therefore $m|(\deg z_0/I(X))$, and $z_0\in \Pic X/m$ has trivial image in $\BBZ/m\BBZ$ in the following exact sequence
\[\Pic^0(X)/m\rightarrow \Pic X/m\xrightarrow{\deg/I(X)} \BBZ/m\BBZ\rightarrow 0\addtag\label{surjective pic mod m}\]
 obtained by $\otimes \BBZ/m\BBZ$ the exact sequence
\[0\rightarrow \Pic^0(X)\rightarrow \Pic X\xrightarrow{\deg/I(X)} \BBZ\rightarrow 0. \]But $\Pic^0(X)/m=0$ by our assumption, we can thus find a global function $a\in K^\times/ K^{\times m}$ such that the valuation of $a$ at $v$ is congruent to $n_v$ modulo $m$, and a fortiori modulo $m_v$ because $m_v|m$. So the class $[a]\in H^1(K,H)$ has image $([a_v])_v$ under the map $H^1(K,H)\rightarrow \oplus_{v\in X^{(1)}}H^1(K_v,H)$. This proves the exactness at the second term.
\par
Given any $l\in \BBZ/m\BBZ$, there exists $\sum_{v\in X^{(1)}}l_v\cdot v\in \Pic X/m$ such that $$\deg (\sum_{v\in X^{(1)}}l_v\cdot v)/I(X)=l\in \BBZ/m\BBZ$$
by the exact sequence (\ref{surjective pic mod m}).
Then $(l_v)_v\in\bigoplus_{v\in X^{(1)}}\BBZ/m_v\BBZ$ has image $l=\BBZ/m\BBZ$ under the map $f_H$, proving the exactness at the third term.
\par
\noindent\textbf{Definition and commutativity of the bottom square.}
By Proposition \ref{comp}, we can identify $\overline H^3_{\nr}(Z,\QZ(2))^D$ with $(\Inv^3(H,\QZ(2))_\rmnorm)^D$ in the sequence. Since $\Inv^3(H,\QZ(2))_\rmnorm$ is a finite cyclic group of order $m$ generated by the Rost invariant $R_H$, we have an isomorphism $(\Inv^3(H,\QZ(2))_\rmnorm)^D\rightarrow \BBZ/m\BBZ$ by evaluating an element at $R_H$. Then the formula (\ref{inv pairing formula}), the definition of $f_H$, and the identification of $\BBZ/m\BBZ$ as the cyclic subgroup of $\BBZ/mI(X)\BBZ$ generated by $I(X)$ altogether give the commutativity.
\vspace{2mm}\par
Now consider $S\neq \emptyset$. We use the identification (\ref{theta mod constant}).
Under the field extension map, the Rost invariant $R_H$ maps to the Rost invariant $R_{H_{K_v}}$ of order $m_v$. Therefore, an element $n_0R_H\in\Inv^3(H,\QZ(2))_\rmnorm$ is in $\overline H^3_{\nr,S}(Z,\QZ(2))$ if and only of $m_v|n_0$ for all $v\in S$, or equivalently $n_0$ is a multiple of $$\lcm(m_v)_{v\in S}=\lcm(\frac{mI(X)}{\gcd(\deg v,mI(X))})_{v\in S}=\frac{mI(X)}{\gcd(I(S),mI(X))}=\frac{m}{\gcd(\frac{I(S)}{I(X)},m)},$$
noting that $I(X)|I(S)$.
Then $$\frac{m}{\gcd(\frac{I(S)}{I(X)},m)}R_H\in \overline H^3_{\nr,S}(Z,\QZ(2))$$ is a generator of this cyclic group of order ${\gcd(\frac{I(S)}{I(X)},m)}$, and evaluation at this generator gives an isomorphism $\overline H^3_{\nr,S}(Z,\QZ(2))^D\rightarrow \BBZ/\gcd(\frac{I(S)}{I(X)},m)\BBZ$.

\par
 Any $(n_v)_v\in\bigoplus_{v\in X^{(1)}\backslash S}\BBZ/m_v\BBZ$ can be completed by elements from $\bigoplus_{v\in S}\BBZ/m_v\BBZ$, whose images under $f_H$ give exactly the subgroup generated by $I(S)/I(X)$ in $\Z/m\Z$. Therefore, the exactness of the right column and also the commutativity follow from the previous case $S=\emptyset$. 

Since the special rational group $E$ is split semisimple simply connected, it satisfies strong approximation away from $S$ (cf. Corollary \ref{split groups satisfy SA}). The rows in the diagram are exact sequences, then a diagram chasing gives the desired result.
\end{proof}

As another application of this reciprocity obstruction to strong approximation, we study classifying varieties of tori. 
\begin{thm}\label{thm rec only one for SLn/tori}Let $Z=E/T$ be a classifying variety of a torus $T$ over $K$ a $p$-adic function field, where the special rational group $E$ is split semisimple simply connected. Then there is an exact sequence of pointed sets
\[1\rightarrow \overline{Z(K)}^S\rightarrow Z(\BA_K^S)\rightarrow \overline H^3_{\nr,S}(Z,\QZ(2))^D\addtag\label{incomplet hoped sequence}\] for $S\subseteq X^{(1)}$ a non-empty finite set of places. In other words,
the reciprocity obstruction to strong approximation away from $S$ is the only one for $Z$. If we suppose furthermore that $\CHOW^2(Z)\rightarrow H^0(K,\CHOW^2(\overline Z))$ is surjective, then there is an exact sequence of pointed sets
\[1\rightarrow \overline{Z(K)}^S\rightarrow Z(\BA_K^S)\rightarrow \overline H^3_{\nr,S}(Z,\QZ(2))^D\rightarrow \overline H^3_{\nr,X^{(1)}}(Z,\QZ(2))^D\rightarrow 1.
\addtag\label{complete exact sequence hoped}\]
\end{thm}
\begin{proof}
We claim that there is a commutative diagram as follows, and the right column is an exact sequence:
\[
\begin{tikzcd}
E(K)\arrow[r]\arrow[d]&Z(K) \arrow[d] \arrow[r] & {H^1(K,T)} \arrow[d] \arrow[r]                 & 1 \\ E(\BA_K^S)      \arrow[r]&Z(\BA_K^S) \arrow[r]\arrow[d]        & \BP_S^1(T) \arrow[r]\arrow[d] & 1\\
&\overline H^3_{\nr,S}(Z,\QZ(2))^D\arrow[r]\arrow [d]&\Sha^1_{X^{(1)}\backslash S}(K,T^\prime)^D\arrow[d]\\
&\overline H^3_{\nr,X^{(1)}}(Z,\QZ(2))^D\arrow[r]\arrow[d]&\Sha^1(K,T^\prime)^D\arrow[d]\\
&1&1
\end{tikzcd}\addtag\label{diagram torus}\]
where $T^\prime$ denotes the dual torus of $T$ (i.e. the torus whose character group is the cocharacter group of $T$), and
\[\Sha^1_{X^{(1)}\backslash S}(K,T):=\Ker(H^1(K,T)\rightarrow\prod_{v\in S}H^1(K_v,T)),\]\[
\Sha^1(K,T):=\Ker(H^1(K,T)\rightarrow\prod_{v\in X^{(1)}}H^1(K_v,T)).\]
Blinstein and Merkurjev gave a description of $\Inv^3(T,\QZ(2))_\rmnorm$ by the exact sequence (cf. Theorem B and Lemma 4.2 of \cite{blinstein2013cohomological})
\[
\begin{split}1\rightarrow \CHOW^2(Z)_\tors \rightarrow H^1(K,T^\prime)\xrightarrow{\alpha} \Inv^3(T,\QZ(2))_\rmnorm\\
\rightarrow H^0(K,\CHOW^2(\overline Z))/\Im(\CHOW^2 (Z))\rightarrow H^2(K,T^\prime).\end{split}\addtag\label{inv of torus sequence}\]
The morphism $\alpha$ is constructed as follows: for every $a\in H^1(K,T^\prime),\ b\in H^1(F,T)$ and every field extension $F/K$, the invariant $\alpha(a)$ sends $b$ to $a_F\cup b$ under the cup-product pairing
\[H^1(F,T^\prime)\otimes H^1(F,T)\rightarrow H^4(F,\BBZ(2))\simeq H^3(F,\QZ(2))\addtag\label{torus pairing}\] which is defined using the quasi-isomorphisms
\[T\simeq\widehat{T^\prime}\otimes^{\BL}\BBZ(1)[1],\quad T^\prime\simeq \widehat T\otimes^\BL \BBZ(1)[1]\addtag\label{quasi isomorphisms}\] 
and pairings \[\widehat{T^\prime}\otimes\widehat T\rightarrow\BBZ,\quad \BBZ(1)[1]\otimes^\BL\BBZ(1)[1]\rightarrow\BBZ(2)[2].\addtag\label{pairings character of tori}\]
On the other hand, there is a Poitou-Tate type exact sequence of topological groups established by Harari, Scheiderer and Szamuely (cf. Theorem 0.1 of \cite{harari2015weak}):
\[\cdots\rightarrow H^1(K,T)\rightarrow \BP^1(T)\rightarrow H^1(K,T^\prime)^D\rightarrow \cdots\addtag\label{Poitou}\] More precisely, the image of the group $H^1(K,T)$ in $\BP^1(T)$ is the right kernel of the pairing 
\[H^1(K,T^\prime)\times \BP^1(T)\xrightarrow{\Sigma}\QZ\] induced by
\[H^1(K_v,T^\prime)\times H^1(K_v,T)\rightarrow H^3(K_v,\QZ(2))\addtag\label{local duality pairing}\] which are defined exactly by using (\ref{torus pairing}). Taking $T$ to be $T^\prime$ in (\ref{Poitou}), and noting that any elements of $\BP_S^1(T^\prime)$ can be completed by $0$'s into an element of $\BP^1(T^\prime)$, we get an exact sequence 
\[ \Sha^1_{X^{(1)}\backslash S}(K,T^\prime)\rightarrow \BP_S^1(T^\prime)\rightarrow H^1(K,T)^D\] which can be completed into the exact sequence 
\[1\rightarrow\Sha^1(K,T^\prime)\rightarrow \Sha^1_{X^{(1)}\backslash S}(K,T^\prime)\rightarrow \BP_S^1(T^\prime)\rightarrow H^1(K,T)^D\addtag\label{before dualizing}\] by definition.
The pairing (\ref{local duality pairing}) is actually a perfect duality, and the subgroups $H^1(\mathcal O_v,\mathcal T)\subseteq H^1(K_v,T)$ and $H^1(\mathcal O_v,\mathcal T^\prime)\subseteq H^1(K_v,T^\prime)$ are exact annihilators of each other (Proposition 1.1 and 1.3 of \cite{harari2015weak}). Therefore, by dualizing (\ref{before dualizing}), we get the exactness of the right column in (\ref{diagram torus}). \par
The commutative diagram
\[\begin{tikzcd}
H^1(K,T^\prime)\arrow[r,"\alpha"]\arrow[d]&\Inv^3(T,\QZ(2))\arrow[d]\\
\prod_{v\in S}H^1(K_v,T^\prime)\arrow[r, "\prod \alpha"]&\prod_{v\in S}\Inv^3(T_{K_v},\QZ(2))
\end{tikzcd}\]
induces a morphism \[
\begin{split}\Sha^1_{X^{(1)}\backslash S}(K,T^\prime)\rightarrow \ker(\Inv^3(T,\QZ(2))_\rmnorm \rightarrow \prod_{v\in S}\Inv^3(T_{K_v},\QZ(2))_\rmnorm)\\
\simeq \overline H^3_{\nr,S}(Z,\QZ(2)).\end{split}\addtag\label{sha induced map}\]
Now applying Proposition \ref{prop inv rec compatible} and the compatibility between the morphism $\alpha$ (\ref{torus pairing}) and the pairing (\ref{local duality pairing}), we get the commutativity of the two squares at the bottom of (\ref{diagram torus}).
\par
Since $E$ satisfies strong approximation away from $S$ (cf. Corollary \ref{split groups satisfy SA}), chasing the diagram (\ref{diagram torus}) gives the desired exactness of (\ref{incomplet hoped sequence}).
\par
Under the further assumption of the surjectivity of $\CHOW^2(Z)\rightarrow H^0(K,\CHOW^2(\overline Z))$, the morphism $\alpha$ is thus surjective by (\ref{inv of torus sequence}), and so is the induced morphism (\ref{sha induced map}). Hence the two bottom rows in (\ref{diagram torus}) are injective. Now a diagram chasing gives the exact sequence (\ref{complete exact sequence hoped}).  
\end{proof}
\begin{eg}
We consider the torus $T^\prime=R_{L/K}(\BBG_{m,L})/\gm$ in Example 4.14 of \cite{blinstein2013cohomological}, where $L/K$ is a degree $n$ field extension induced by a continuous surjective group morphism from the absolute Galois group of $K$ to the symmetric group $S_n$. For example, the generic maximal torus of the group $\PGL_n$ is of this form. The dual torus $T$ of $T^\prime$ is the norm one torus $R_{L/K}^{(1)}(\BBG_{m,L}).$ Let $Z=E/T$ be a classifying variety. Then $H^0(K,\CHOW^2(\overline Z))/\Im(\CHOW(Z))$ is trivial, and we have $\overline H^3_{\nr}(Z,\QZ(2))\simeq \br(L/K)$. Since local-global principle holds for the Brauer group of $K$  (cf. Remark \ref{local-global for brauer}), the group $\overline H^3_{\nr,X^{(1)}}(Z,\QZ(2))$ is trivial. We apply (\ref{complete exact sequence hoped}) and get the exact sequence of pointed sets
\[1\rightarrow \overline{Z(K)}^S\rightarrow Z(\BA_K^S)\rightarrow \ker(\br(L/K)\rightarrow\prod_{v\in S}\br(K_v\otimes L/K_v))^D\rightarrow 1. \]
\end{eg}
\begin{rmk}
The exact sequence in Theorem \ref{thm rec only one sc} actually also takes the form (\ref{complete exact sequence hoped}), because $$\overline H^3_{\nr, X^{(1)}}(Z,\QZ(2))\subseteq \overline H^3_{\nr,S^\prime}(Z,\QZ(2))=1$$ for some finite set $S^\prime$ such that $I(S^\prime)=I(X).$ 
Therefore, we may hope this exact sequence to hold in other more general situations.
\end{rmk}
In fact, the commutative diagram in the proof of Proposition \ref{prop inv rec compatible} showing local compatibility is generally true for $\Inv^d(H,\QZ(d-1))$ and $H^d_\nr(Z,\QZ(d-1))$ over any field of characteristic $0$ (while in characteristic $p$, we can identify $\Inv^d(H,\QZ(d-1))$ with the subgroup of ``balanced elements'' in $H^0_\Zar(Z,\mathcal H ^d(\QZ(d-1)))$ under certain conditions, cf. Theorem A of \cite{blinstein2013cohomological}). In particular, taking $d=2$, we get the group $\Inv(H,\br)$ of invariants with values in the Brauer group, and $$H^0_\Zar(Z,\mathcal H^2(\QZ(1)))=H^2(Z,\QZ(1))=\br Z.$$ This enables us to relate $\Inv(H,\br)$ with the Brauer-Manin pairing, which is used for studying the arithmetic of varieties over a number field $F$, for which we get the following commutative diagram
\[
\begin{tikzcd}
Z(\BA_K) \arrow[d,Rightarrow, no head] \arrow[r,phantom,"\times" description] & \Inv(H,\br) \arrow[r] \arrow[d, "\simeq"]       & \QZ \arrow[d,Rightarrow, no head] \\
Z(\BA_K)                                \arrow[r,phantom,"\times" description] & {\br Z} \arrow[r]  & \QZ.     
\end{tikzcd}\addtag\label{cd=2 diagram pairing compatibility}
\]
Now we explain how this is compatible with known results of strong approximation over $F$ a number field.\par 
For $H$ a connected linear group over a field $K$ of characteristic $0$, Blinstein and Merkurjev (cf. Theorem 2.4 of \cite{blinstein2013cohomological}) showed that 
\[\Pic H\oplus \br K\xrightarrow[~\simeq~]{\nu_1}\Inv(H,\br)\simeq\br Z\]
whose construction and proof are based on the exact sequence 
\[ \Pic W \rightarrow \Pic H\rightarrow\br Y\rightarrow \br W\addtag\label{Pic Pic Br Br exact sequence} \]
associated to a smooth $K$-variety $Y$ and a $Y$-torsor $W$ under $H$, as defined in Proposition 6.10 of \cite{sansuc1981groupe}. In fact, if we use the exact sequence in Theorem 2.8 of \cite{borovoi2013manin} which is of the same form as (\ref{Pic Pic Br Br exact sequence}) but with $\Pic H\rightarrow \br Y$ constructed in a different way via the abelian group $\ext_F^c(H,\gm)$ of isomorphism classes of central extensions of $K$-algebraic groups of $H$ by $\gm$, we can still get an isomorphism \[\Pic H\oplus \br K\xrightarrow[~\simeq~]{\nu_2}\Inv(H,\br)\simeq\br Z\addtag\label{Br Z = Inv = Pic + Br F}\] with the same proof as in Theorem A of \cite{blinstein2013cohomological}. 
\par
Now for $F$ a number field, the above diagram (\ref{cd=2 diagram pairing compatibility}) along with the identification (\ref{Br Z = Inv = Pic + Br F}) gives rise to the following commutative diagram, where the right column is an exact sequence of pointed sets (cf. Corollary 2.5 and Proposition 2.6 of \cite{kottwitz1986stable}):
 \[
\begin{tikzcd}
E(F)\arrow[r]\arrow[d]&Z(F) \arrow[d] \arrow[r] & {H^1(F,H)} \arrow[d] \arrow[r]                 & 1 \\ E(\BA_F)      \arrow[r]&Z(\BA_F) \arrow[r]\arrow[d]        & \BP^1(H) \arrow[r]\arrow[d] & 1\\
&(\br Z/\Im(\br F))^D\arrow[r, "\simeq"]&(\Pic H)^D.\\
\end{tikzcd}\]
For $S\subseteq \Omega_F$ such that $E$ satisfies strong approximation away from $S$ (e.g. $E=\SL_n$ with any non-empty finite set $S$), a diagram chasing argument shows that the Brauer-Manin obstruction to strong approximation away from $S$ (defined using projection and the whole Brauer group, cf. Remark \ref{rmk new def of obstruction})  is the only one for the classifying variety $Z$. See Theorem 3.7 of \cite{colliot-xu2009integral} where these results were obtained for the first time.
\par Now we apply Demeio's version of obstruction given by $\br _S Z$ (see Remark \ref{rmk new def of obstruction}). Let $S$ be a set containing only non-archimedean places and denote by $Z(\BA_F)_\bullet$ the adelic points where each archimedean component is collapsed to the (discrete) topological space of its connected components. We have the following commutative diagram
\[
\begin{tikzcd}
E(F)\arrow[r]\arrow[d]&Z(F) \arrow[d] \arrow[r] & {H^1(F,H)} \arrow[d] \arrow[r]                 & 1 \\ E(\BA_F^S)_\bullet      \arrow[r]&Z(\BA_F^S)_\bullet \arrow[r]\arrow[d]        & \BP_S^1(H) \arrow[r]\arrow[d] & 1\\
&(\br_S Z/\Im(\br_S F))^D\arrow[r, "\simeq"]&(\ker(\Pic H\rightarrow\prod_{v\in S}\Pic H_{F_v}))^D
\end{tikzcd}\addtag\label{diagram BM}\]
where the bottom row is induced by (\ref{Br Z = Inv = Pic + Br F}).

The right column is an exact sequence of pointed sets, which can be seen by chasing the following commutative diagram (cf. Theorem 1.2 of \cite{kottwitz1986stable} for the right-most bijection):
\[
\begin{tikzcd}
H^1(F,H)\arrow[r,equal]\arrow[d]& H^1(F,H)\arrow[d]\\
 \BP_S^1(H)\arrow[d]&\BP^1(H)\arrow[l,two heads]\arrow[d]&\prod_{v\in S}H^1(F_v,H)\arrow[l,hook']\arrow[d,"\simeq"]\\
 (\ker(\Pic H\rightarrow\prod_{v\in S}\Pic H_{F_v}))^D &(\Pic H)^D\arrow[l,two heads]& (\prod_{v\in S}\Pic H_{F_v})^D\arrow[l,hook'].
\end{tikzcd}\addtag\label{diagram removing S in PT sequence, Pic case}
\]
 If $E(F)$ is dense in $E(\BA_F^S)_\bullet$ (for example $E=\SL_n$), then chasing the diagram (\ref{diagram BM}) gives $\overline{Z(F)}^S=Z(\BA_F^S)_\bullet^{\br_S Z}$, meaning that the Brauer-Manin obstruction to strong approximation away from $S$ is the only one for $Z$. See Theorem 6.1 and Remark 6.3 of Demeio's work \cite{demeio2022etale} for the general result.
\subsection{Application to tori}
For a torus over a $p$-adic function field, strong approximation is far from being satisfied. In fact, Harari and Izquierdo (cf. \cite{harari2019espace}) showed that for $K$ the function field of a smooth projective $X$ over any field $k$ of characteristic $0$, any finite set $S\subseteq X^{(1)}$, and any $K$-torus $T$, the image of $T(K)$ in $T(\BA_K^S)$ is a discrete (hence closed) subgroup. In particular, strong approximation never holds for a torus of dimension $>0.$ The group $A(T):=T(\BA_K)/\overline{T(K)}$ measuring the defect of strong approximation is thus the same as $T(\BA_K)/T(K)$. With the exact sequence
\[1\rightarrow H^0(K,T)_\wedge\rightarrow \BP^0(T)_\wedge\rightarrow H^2(K,T^\prime)^D\rightarrow \Sha^2(K,T^\prime)^D\rightarrow 1\addtag\label{Poitou-Tate H0}\]
from the Poitou-Tate type exact sequence in \cite{harari2015weak}, where the pairing \[\BP^0(T)\times H^2(K,T^\prime)\rightarrow \QZ\] comes from the cup-products induced by $T\otimes T^\prime\rightarrow \Z(2)[2]$ (cf. (\ref{quasi isomorphisms}) and (\ref{pairings character of tori})), Harari and Izquierdo described $A(T)$ when $K$ is a $p$-adic function field:
\begin{thm}[Corollaire 6.7 and Corollaire 6.9 of \cite{harari2019espace}]\label{hi19 result}
There is an injective morphism 
\[{A(T)/\Div}\hookrightarrow H^2(K,T^\prime)\]
where ${A(T)/\Div}$ denotes the quotient of $A(T)$ by its maximal divisible subgroup. The closure $J$ of its image fits into an exact sequence
\[1\rightarrow J\rightarrow H^2(K,T^\prime)^D\rightarrow\Sha^2(T^\prime)^D\rightarrow 1.\]
There is a perfect pairing
\[A(T)_\tors\times\frac{H^2(K,T^\prime)_\wedge}{\Im(\Sha^2(K,T^\prime))}\rightarrow \QZ.\]

\end{thm}

Now we explain that this is compatible with our reciprocity obstruction given by $H^3_\nr(T,\QZ(2))$.
\begin{prop}\label{prop cup product commutes with rec obs}
There is a morphism $H^2(K,T^\prime)\rightarrow H^3_\nr(T,\QZ(2))$ which fits into the following commutative diagram
\[\begin{tikzcd}
{H^0(K_v,T)} \arrow[r,phantom, "\times" description] \ar[equal,d] & {H^2(K,T^\prime)} \arrow[d,] \arrow[r]           & {H^3(K_v,\QZ(2))} \arrow[equal,d] \\
T(K_v) \arrow[r,phantom, "\times" description]                                      & {H^3_\nr(T,\QZ(2))} \arrow[r, "(\ref{pairing at one place})"]                   & {H^3(K_v,\QZ(2)).}                               
\end{tikzcd}\]
\end{prop}
\begin{proof}
Given an element $a:\Spec K_v\rightarrow T$ in $H^0(K_v,T)$, the functoriality of the cup-products induced by $T\otimes T^\prime\rightarrow \Z(2)[2]$ gives the following commutative diagram:
\[
\begin{tikzcd}
H^0(K_v,T)\arrow[r,phantom, "\times" description] & H^2(K,T^\prime)\arrow[d,equal]\arrow[r]& H^4(K_v,\Z(2))\\
H^0(T,T)\arrow[u,"a^*"]\arrow[r,phantom,"\times" description]& H^2(K, T^\prime)\arrow[r]& H^4(T,\Z(2))\arrow[u,"a^*"],
\end{tikzcd}
\]
where $a^*$ denotes the pullback maps induced by $a$. The identity $\Id\in H^0(T,T)$ is mapped to $a$ under $a^*$, and thus we get the commutativity of the upper part of the following diagram:
\[\begin{tikzcd}
{H^0(K_v,T)} \arrow[r,phantom, "\times" description] \ar[equal,d] & {H^2(K,T^\prime)} \arrow[d, "\Id\cup -"] \arrow[r]           & {H^4(K_v,\BBZ(2))} \arrow[equal,d] \\
{T(K_v)} \arrow[r,phantom, "\times" description] \arrow[d, equal] & {H^4(T,\BBZ(2))} \arrow[d, "{(\ref{kahn map})}"] \arrow[r,"\text{evaluation}"] & {H^4(K_v,\BBZ(2))} \arrow[d,leftarrow, "\simeq"]                      \\
T(K_v) \arrow[r,phantom, "\times" description]                                      & {H^3_\nr(T,\QZ(2))} \arrow[r, "(\ref{pairing at one place})"]                   & {H^3(K_v,\QZ(2))}                               
\end{tikzcd}\addtag\label{diagram cup product pairing}\]
where the morphism $H^2(K,T^\prime)\rightarrow H^4(T,\Z(2))$ is given by taking the cup-product with $\Id\in H^0(T,T)$.
\par Now we construct and prove the commutativity of the lower part of this diagram. Recall that (cf. Proposition 2.9 of \cite{kahn2012classes}) for a smooth $K$-variety $Z,$ we have a natural map $H^4(Z,\Z(2))\rightarrow H^3_{\nr}(Z,\QZ(2))$ 
defined as follows. Let $\alpha: Z_\et\rightarrow Z_{\Zar}$ be the change-of-sites map. The natural map $\BBQ(2)_\Zar\rightarrow R\alpha_*\BBQ(2)$ is an isomorphism in the derived category of Zariski sheaves (cf. Lemma 2.5 and Theorem 2.6 of \cite{kahn2012classes}). Since $\BBQ(2)_\Zar$ is concentrated in degree $\leq 2$, we have $R^4\alpha_*\BBQ(2)=R^3\alpha_*\BBQ(2)=0$, hence $R^4\alpha_*\Z(2)\simeq R^3\alpha_*\QZ(2)$. Then the Leray spectral sequence for $\alpha$ yields an edge map
\[H^4(Z,\Z(2)\rightarrow H^0_{\Zar}(Z,R^4{\alpha_*}\Z(2))\simeq H^0_{\Zar}(Z,R^3\alpha_*\QZ(2))\overset{\text{\ref{def}}}{=} H^3_{\nr}(Z,\QZ(2)).\addtag\label{kahn map}\]
The construction of the Gersten resolution implies we have the commutative diagram
\[\begin{tikzcd}
{H^4(Z,\BBZ(2))} \arrow[d, "(\ref{kahn map})"] \arrow[r] & H^4(K(Z),\BBZ(2))\arrow[d,leftarrow, "\simeq"]\\
{H^3_\nr(Z,\QZ(2))}\arrow[r,hook]& H^3(K(Z),\QZ(2))
\end{tikzcd}\]
 where $H^4(Z,\BBZ(2))\rightarrow H^4(K(Z),\BBZ(2))$ is the restriction map. This gives the commutativity of the lower part of (\ref{diagram cup product pairing}). 
\end{proof}
\begin{cor}\label{cor SA tori}
The morphism $T(\BA_K)\rightarrow H^3_{\nr}(T,\QZ(2))^D$ induces injective morphisms
\[{A(T)/\Div}\hookrightarrow H^3_{\nr}(T,\QZ(2))^D,\]
\[A(T)_\tors\hookrightarrow ({H^3_{\nr}(T,\QZ(2))_\wedge})^D.\]
\end{cor}
\begin{proof}
Proposition \ref{prop cup product commutes with rec obs} along with (\ref{Poitou-Tate H0}) gives the exact sequence
\[1\rightarrow T(K)\rightarrow T(\BA_K)\rightarrow H^3_\nr(T,\QZ(2))^D.\]
Now apply Lemme 2.7 and Théorème 6.6 of \cite{harari2019espace}, and we get the first injection. \par
The morphism constructed in Proposition \ref{prop cup product commutes with rec obs} gives $H^2(K,T^\prime)/m\rightarrow H^3_\nr(T,\QZ(2))/m$ for all $m>0$, and thus a morphism $H^2(K,T^\prime)_\wedge \rightarrow H^3_\nr(T,\QZ(2))_\wedge$, compatible with the pairing in Theorem \ref{hi19 result}, yielding the second injection.
\end{proof}

\section{Weak approximation for classifying varieties}
In the study of weak approximation problems over $K$ a $p$-adic function field, we can define a pairing (cf. \cite{harari2015weak})
\[\prod_{v\in X^{(1)}}Z(K_v)\times H_\nr^3(K(Z)/K,\QZ(2))\rightarrow\QZ\addtag\label{WA pairing}\]
whose left kernel contains the closure $\overline {Z(K)}$ inside $\prod_{v\in X^{(1)}}Z(K_v)$ with respect to the product topology, defining the reciprocity obstruction to weak approximation.
\par
For $Z=E/H$ a classifying variety where the special rational group $E$ is split semisimple simply connected, we can establish the same compatibility between the reciprocity obstruction and the cohomological invariants.\par
An invariant $I\in \Inv^3(H,\QZ(2))$ is called \textit{unramified} if for every field extension $F/K$ and every element $H^1(F,H)$, we have $I(a)$ in $H^3_\nr(F/K,\QZ(2))$ the group of elements unramified with respect to all discrete valuations of $F$ trivial on $K$. 
Under the identification (\ref{theta}), the subgroup $H^3_\nr(K(Z)/K,\QZ(2))$ of $H^3_\nr(Z,\QZ(2))$ is identified with the group $\Inv^3_\nr(H,\QZ(2))$ of unramified invariants (cf. Proposition 4.1 of \cite{merkurjev2016reductive}). We denote by $\Inv^3_\nr(H,\QZ(2))_\rmnorm$ the group of unramified normalized invariants. 
\begin{prop}\label{WA analogue pairing prop}
There is a well-defined pairing 
\begin{align*}
\prod_{v\in X^{(1)}}Z(K_v)\times\Inv_\nr^3(H,\QZ(2))&\rightarrow\QZ\\
((x_v)_v,I) &\mapsto \sum_{v\in X^{(1)}}\Cor_{\kappa(v)/k}(\partial_v(I_{K_v}(x_v^*E)))\nonumber\end{align*}
induced by evaluating an invariant $I$ at the $H$-torsors $x_v^*E$ over $K_v$, where $x_v^*E$ denotes the fiber of the $H$-torsor $E\rightarrow Z$ above the point $x_v\in Z(K_v)$. 
This pairing is compatible with (\ref{WA pairing}) in the sense of the following commutative diagram 
\[
\begin{tikzcd}
\prod_{v\in X^{(1)}}Z(K_v) \arrow[d,Rightarrow, no head] \arrow[r,phantom,"\times" description] & {\Inv_\nr^3(H,\QZ(2))_\rmnorm} \arrow[r]        & \QZ \arrow[d,Rightarrow, no head] \\
\prod_{v\in X^{(1)}}Z(K_v)                              \arrow[r,phantom,"\times" description] & {H^3_\nr(K(Z)/K,\QZ(2))/\Im(H^3(K,\QZ(2)))} \arrow[r] \arrow[u, leftarrow,"\simeq","\theta"' ] & \QZ.
\end{tikzcd}\addtag\label{unramified pairing}\]
\end{prop}
\begin{proof}
Same as that of Proposition \ref{prop inv rec compatible}.
\end{proof}
As an application, we show how this method gives answers to weak approximation problems for classifying varieties of tori over $K$. The following result was also obtained by Linh in a different way in his very recent work (cf. Theorem B of \cite{linh2022arithmetics}). Now we give our proof using cohomological invariants.
\begin{thm}\label{thm WA SLn/T}
    Let $Z=E/T$ be a classifying variety of a torus $T$ over $K$. Then the reciprocity obstruction to weak approximation is the only one for $Z$. In fact, there is a morphism $$\Sha_S^1(T^\prime)\rightarrow H^3_\nr(K(Z)/K,\QZ(2))/\Im(H^3(K,\QZ(2)))$$ such that the subset of elements in $\prod_{v\in S}Z(K_v)$ orthogonal to the image of $\Sha_S^1(T^\prime)$ in $\frac{H^3_\nr(K(Z)/K,\QZ(2))}{\Im(H^3(K,\QZ(2))) }$ with respect to the pairing (\ref{WA pairing}) already equals the closure $\overline {Z(K)}$ in the topological product $\prod_{v\in S}Z(K_v)$, where $\Sha^1_S(T^\prime):=\Ker(H^1(K,T^\prime)\rightarrow\prod_{v
\in X^{(1)}\backslash S}H^1(K_v,T^\prime))$ for a finite set $S$.
\end{thm}
\begin{proof}
Blinstein and Merkurjev (cf. Theorem 5.5 of \cite{blinstein2013cohomological}) gave a natural isomorphism 
\begin{align*}\Inv^3(T_1,\QZ(2))&\simeq \Inv^3_\nr(T,\QZ(2))\\
I_1&\mapsto (b\mapsto I_1(g_*(b)))
\end{align*} where $T_1$ is a flasque torus, fitting into the exact sequence $1\rightarrow T\xrightarrow{g} T_1\rightarrow P\rightarrow 1$
with $P$ a quasi-trivial torus (such a resolution for a given $T$ always exists by Lemma 0.6 of \cite{colliot1987principal}).  
 We have its dual exact sequence $1\rightarrow P^\prime\rightarrow T_1^\prime\rightarrow T^\prime\rightarrow 1$
where $P^\prime$ being the dual of a quasi-trivial torus is still quasi-trivial. The induced exact sequence 
\[1\rightarrow H^1(K,T_1^\prime)\rightarrow H^1(K,T^\prime)\rightarrow H^2(K,P^\prime)\] gives an isomorphism $\Sha_S^1(T_1^\prime)\simeq \Sha_S^1(T^\prime)$ because $\Sha^2_S(P^\prime)=0$ by Lemma 3.2 of \cite{harari2015weak}. Consider the composite map 
\[\begin{split}\Sha_S^1(T^\prime)\simeq\Sha_S^1(T_1^\prime)\hookrightarrow H^1(K,T_1^\prime)\xrightarrow[(\ref{torus pairing})]{\alpha} \Inv^3(T_1,\QZ(2))_\rmnorm \\
\simeq \Inv^3_{\nr}(T,\QZ(2))_\rmnorm \simeq H^3_\nr(K(Z)/K,\QZ(2))/\Im(H^3(K,\QZ(2)))\end{split}\addtag\label{long iso construction WA tori}\]
which gives rise to the map at the bottom of the following diagram
\[
\begin{tikzcd}
E(K)\arrow[r]\arrow[d]&Z(K) \arrow[d] \arrow[r] & {H^1(K,T)} \arrow[d] \arrow[r]                 & 1 \\ \prod_{v\in S}E(K_v)      \arrow[r]&\prod_{v\in S}Z(K_v) \arrow[r]\arrow[d]        & \prod_{v\in S} H^1(K_v,T) \arrow[r]\arrow[d] & 1\\
&(\frac{H^3_\nr(K(Z)/K,\QZ(2))}{\Im(H^3(K,\QZ(2))))})^D\arrow[r]&\Sha^1_S(T^\prime)^D.\\
\end{tikzcd}\addtag\label{diagram WA SLn/torus}\]
 The right column of the diagram comes from the Poitou-Tate type exact sequence (\ref{Poitou}) and is exact by the same argument as in the proof in Proposition \ref{thm rec only one for SLn/tori}.
 Now we check that the bottom square of (\ref{diagram WA SLn/torus}) is commutative. In virtue of Proposition \ref{WA analogue pairing prop}, it suffices to show the commutativity of the following diagram
 \[\begin{tikzcd}
{H^1(K_v,T)} \arrow[r,phantom, "\times" description] \ar[equal,d] & {\Sha_S^1(T^\prime)} \arrow[d,"(\ref{long iso construction WA tori})"] \arrow[r, "(\ref{torus pairing})"]           & \QZ \arrow[equal,d] \\
H^1(K_v,T) \arrow[r,phantom, "\times" description]                                      & \Inv^3_\nr(T,\QZ(2))\arrow[r]                   & \QZ.                          
\end{tikzcd}\]
Let $a\in\Sha^1_S(T^\prime)$. Let $a_1\in H^1(K,T_1^\prime),I_1\in\Inv^3(T_1,\QZ(2))_\rmnorm, I\in\Inv^3_\nr(T,\QZ(2))_\rmnorm$ be the images of $a$ through the consecutive maps (\ref{long iso construction WA tori}). For all $b\in H^1(K_v,T)$, we have
\[I(b)=I_1(g_*(b))= a_1\cup g_*(b)=a\cup b \]
where $I(b)$ (resp. $a\cup b$) gives exactly the pairing between $b$ and $I$ (resp. $a$) with respect to the lower (resp. upper) row, yielding the commutativity.
\par
The special rational group $E$ satisfies weak approximation, then chasing the diagram (\ref{diagram WA SLn/torus}) gives the desired result.
\end{proof}
It is still an open question whether weak approximation holds for $Z=\SL_n/H$ where $H$ is a semisimple simply connected group over $K$ a $p$-adic function field. In case of a negative answer, we might hope to find a reciprocity obstruction using $H^3_\nr(K(Z)/K,\QZ(2))\simeq \Inv^3_\nr (H,\QZ(2))$. Merkurjev studied $\Inv^3_\nr(H,\QZ(2))$ for all classical semisimple simply connected groups $G$, and he showed in \cite{merkurjev2002unramified} that $\Inv^3_\nr(H,\QZ(2))_\rmnorm$ is trivial except for $H$ of type $\prescript{2}{}{\mathsf A}_{n-1},\prescript{2}{}{\mathsf D_3}$ or $\prescript{1}{}{\mathsf D_4}$ under certain conditions, where $\Inv^3_\nr(H,\QZ(2))_\rmnorm\simeq\Z/2\Z$. Garibaldi computed the remaining cases when $H$ is \textit{exceptional}, i.e. when $H$ is of type $\mathsf G_2,\prescript{3}{}{\mathsf D}_4,\prescript{6}{}{\mathsf D}_4,\mathsf F_4,\mathsf E_6,\mathsf E_7$ or $\mathsf E_8$, and he showed in \cite{garibaldi2006unramified} that for $H$ a simple simply connected exceptional algebraic group, the group $\Inv^3_\nr(G,\QZ(2))_\rmnorm$ is $\Z/2\Z$ if $H$ is of type $\prescript{3}{}{\mathsf D}_4$ with a nontrivial Tits algebra; otherwise $\Inv^3_\nr(H,\QZ(2))_\rmnorm$ is trivial. It would be interesting to see if weak approximation could fail for $\SL_n/H$ in the case of a non-trivial $\Inv^3_\nr(H,\QZ(2))_\rmnorm$.

\bibliographystyle{alpha}
\bibliography{bibliographie}

\end{document}